\documentclass[11pt]{amsart}
\usepackage{amsmath,amssymb,euscript,mathrsfs,pstricks}
\usepackage[small,heads=LaTeX,nohug]{diagrams}
\usepackage{hyperref}
\usepackage[all]{xy}
\usepackage{enumerate}

\pushQED{\qed}

\psset{unit=1pt}
\psset{arrowsize=4pt 1}
\psset{linewidth=.5pt}

\newcommand{\define}{\textbf}
\newcommand{\comment}{$\star$ \texttt}

\newcommand{\isom}{\cong}
\renewcommand{\setminus}{\smallsetminus}
\renewcommand{\phi}{\varphi}

\renewcommand{\tilde}{\widetilde}

\renewcommand{\bar}{\overline}
\renewcommand{\wedge}{\bigwedge}

\newcommand{\C}{\mathbb{C}}

\newcommand{\R}{\mathbb{R}}
\newcommand{\N}{\mathbb{N}}
\newcommand{\Z}{\mathbb{Z}}
\renewcommand{\P}{\mathbb{P}}

\newcommand{\mth}{\mathrm{th}}

\newcommand{\lef}{\left\langle}
\newcommand{\rig}{\right\rangle}

\DeclareMathOperator{\Sym}{Sym}

\DeclareMathOperator{\tr}{tr}

\DeclareMathOperator{\Hom}{Hom}
\DeclareMathOperator{\Spec}{Spec}

\DeclareMathOperator{\Ind}{Ind}
\DeclareMathOperator{\prim}{prim}
\DeclareMathOperator{\Int}{Int}
\DeclareMathOperator{\Proj}{Proj}
\DeclareMathOperator{\class}{class}
\DeclareMathOperator{\st}{st}
\DeclareMathOperator{\vol}{vol}

\DeclareMathOperator{\BOX}{Box}

\DeclareMathOperator{\conv}{conv}
\DeclareMathOperator{\sgn}{sgn}
\DeclareMathOperator{\coker}{coker}

\newtheorem{theorem}{Theorem}[section]
\newtheorem*{ntheorem}{Theorem}
\newtheorem{lemma}[theorem]{Lemma}

\newtheorem{proposition}[theorem]{Proposition}

\newtheorem{corollary}[theorem]{Corollary}
\newtheorem*{ncor}{Corollary}

\theoremstyle{definition}
\newtheorem{definition}[theorem]{Definition}
\newtheorem{remark}[theorem]{Remark}
\newtheorem{example}[theorem]{Example}

\newtheorem{conjecture}[theorem]{Conjecture}
\newtheorem*{nconj}{Conjecture}

\newcommand{\excise}[1]{}

\begin{document}

\title{Representations on the cohomology of hypersurfaces and mirror symmetry}
\author{Alan Stapledon}
\address{Department of Mathematics\\University of British Columbia\\ BC, Canada V6T 1Z2}
\email{astapldn@ubc.math.ca}


\date{April 18, 2010} 
\thanks{
Part of this work was completed while visiting Sydney University, with funding from
the Australian Research Council project DP0559325, Chief Investigator Professor G I Lehrer. 
}


\begin{abstract}
We study the representation of a finite group acting on the cohomology of a non-degenerate, invariant hypersurface of a projective toric variety. We deduce an explicit description of the representation when the toric variety has at worst quotient singularities. As an application, we conjecture a representation-theoretic version of Batyrev and Borisov's mirror symmetry between pairs of Calabi-Yau
hypersurfaces, and prove it when the hypersurfaces are both smooth or have dimension at most $3$. 
An interesting consequence is the existence of pairs of Calabi-Yau orbifolds whose Hodge diamonds are mirror, with respect to the usual Hodge structure on singular cohomology. 
\end{abstract}

\maketitle

\section{Introduction}

 When a finite group $G$ acts algebraically on a complex variety $Z$, it is an important problem to determine the corresponding representation of $G$ on the complex cohomology $H^*Z$ of $Z$.
 In particular, if $Z$ is complete and has at worst quotient singularities, then 
 the Hodge structure of the cohomology of $Z/G$ is determined by the isomorphism $H^* (Z/G)  \cong (H^* Z)^G$. 
 We refer the reader to the work of Dimca and Leher \cite{DLPurity}, 
 Cappell, Maxim, Schuermann, Shaneson  \cite{CLSSEquivariant, CMSEquivariant}, and Ch\^enevert \cite{CheRepresentations} for recent developments on this topic. In the case when $Z$ is a toric variety associated to root system and $G$ is the associated Weyl group, the corresponding representation $H^*Z$ has been studied by  Procesi~\cite{ProToric}, Stanley \cite[p. 529]{StaLog}, Dolgachev, Lunts~\cite{DLCharacter}, Stembridge~\cite{SteSome, SteEulerian} and Lehrer~\cite{LehRational}. 
The purpose of this article is to study the representation $H^*Z$ in the case when $Z$ is an invariant hypersurface of a toric variety. 

Let $G$ be a finite group with representation ring $R(G)$. 
Let $\rho: G \rightarrow GL(M)$ be a linear action of $G$ on a lattice $M \cong \Z^d$, and 
consider the corresponding action of $G$ on the torus $T = \Spec \C[M]$. Let 
$X^\circ = \{ \sum_{u \in M} a_u \chi^u = 0 \} \subseteq T$ be a $G$-invariant hypersurface which is \define{non-degenerate} with respect to its Newton polytope $P = \conv\{ u \mid a_u \ne 0 \}$ (see Section~\ref{toric}).
Then the normal fan to $P$ determines a projective toric variety $Y = Y_P$, and the action of $G$ on $T$ extends to an action of $G$ on $Y$ via toric morphisms. The closure $X$ of $X^\circ$ in $Y$ is a $G$-invariant, projective variety.


For any 
complex variety $Z$ with $G$-action, we introduce the \define{equivariant Hodge-Deligne polynomial} $E_G(Z;u, v) = \sum_{p,q} 
e^{p,q}_G u^p v^q  \in R(G)[u,v]$ of $Z$ (see Section~\ref{s:HodgeDeligne}), satisfying the following 
properties 
\begin{enumerate}
\item\label{i:1} if $U$ is a $G$-invariant open subvariety of $Z$, then \\ $E_G(Z)  = E_G(U) + E_G(Z \setminus U)$, 
\item\label{i:2} if $Z$ is complete and has at worst quotient singularities, then \\
$E_G(Z) = \sum_{p,q} (-1)^{p + q} H^{p,q}(Z)  u^p v^q$. 
\end{enumerate}
This generalizes the usual notion of Hodge-Deligne polynomial when $G$ is trivial, and reduces to both the \emph{weight polynomial} 
$E_G(Z; t, t)$ of Dimca and Lehrer \cite{DLPurity}, and the  \emph{equivariant $\chi_y$-genus} $E_G(Z; u, 1)$ of Cappell, Maxim and Shaneson \cite{CMSEquivariant}. 
Our first main result is an explicit algorithm to determine $E_G(X^\circ;u, v)$. 
We refer the reader to  Section~\ref{s:ehd} for details. 
In particular, the algorithm determines the representations of $G$ on the pieces of the mixed Hodge structure of the cohomology of $X^\circ$ with compact support (Remark~\ref{r:Hodge-Deligne}). 
By the additivity property \eqref{i:1}, 
one can then inductively compute $E_G(X)$, and  hence, by \eqref{i:2}, we deduce the representations
of $G$ on the  $(p,q)$-pieces of the cohomology of $X$. 

In order to state our results more precisely, we recall a combinatorial construction which was introduced and studied in \cite{YoEquivariant}. 
For any positive integer $m$, $G$ permutes the lattice points in the $m^{\mth}$ dilate of $P$, and we may consider the corresponding permutation representation $\chi_{mP}$. 
Motivated precisely by the computations in this paper, the author introduced  a power series 
of virtual representations $\phi[t] = \sum_{i \ge 0} \phi_i t^i \in R(G)[[t]]$, 
determined by the equation 
\begin{equation*}
1 + \sum_{m \ge 1} \chi_{mP} t^m = \frac{ \phi[t]}
{(1- t)(1 - \rho \, t + \wedge^2  \rho \,  t^2 - \cdots + (-1)^d \wedge^d \rho \,  t^d)}.  
\end{equation*}
While the power series $\phi[t]$ is not a polynomial for  general $G$ and $P$ (see \cite[Section~7]{YoEquivariant}), 
we prove that the existence of a $G$-invariant, non-degenerate hypersurface with Newton polytope $P$ implies that $\phi[t]$ is a polynomial, and the virtual representations $\phi_i$ are effective representations (Corollary~\ref{c:vanishing}). 
If we let $\det(\rho) = \bigwedge^{d} \rho$, then the theorem below computes the equivariant $\chi_y$-genus $E_G(X^\circ; u, 1)$ of $X^{\circ}$.


\begin{ntheorem}[Theorem~\ref{t:xycharacteristic}]
For any $p \ge 0$, 
\[
\sum_q e^{p,q}_G(X^\circ) = (-1)^{d - 1 - p} \bigwedge^{d - 1 - p} \rho + (-1)^{d - 1}\det(\rho) \cdot \phi_{p + 1}. 
\]
\end{ntheorem}

In Theorem~\ref{t:simpleint}, we produce an explicit formula for $E_G(X^\circ)$ in the case when $P$ is \define{simple} i.e.
when every vertex of $P$ is contained in precisely $d$ facets. 
In particular, for $p > 0$, we show that $(-1)^{d - 1} \det(\rho) \cdot e_G^{p,0}(X^{\circ})$ equals the permutation representation induced by the action of $G$ on the lattice points which lie in the relative interior of a $(p + 1)$-dimensional face of $P$ (Corollary~\ref{c:p0}).

The condition that $P$ is simple is equivalent to the requirement that the toric variety $Y$ has at worst quotient singularities. In this case, 
$X$ has at worst quotient singularities, and 
computing the representation of $G$ on $H^* X$ reduces to computing the representation of $G$ on the \define{primitive cohomology} 
$H^{d - 1}_{\prim} X = \bigoplus_{p = 0}^{d} H^{p, d - 1 - p}_{\prim}(X)$ (see Section~\ref{s:cohomology}).
  In fact, we have isomorphisms  of $G$-representations $H^{p, d - 1 - p}_{\prim}(X) \cong H^{d - 1 - p,   p}_{\prim}(X)$ (Remark~\ref{r:symmetry}), and hence we may reduce to the case when $p \ge \frac{d - 1}{2}$.  
For any face $Q$ of $P$, let $G_Q$ denote the isotropy subgroup of $Q$. In Section~\ref{s2}, 
we define a representation $\rho_Q: G_Q \rightarrow GL(M^Q)$, where $M^Q$ is a translation of the intersection of the affine span of $Q$ with $M$. 


\begin{ntheorem}[Theorem~\ref{t:simple}]
If $P$ is simple and $p \ge \frac{d - 1}{2}$, then
\[
H^{p, d - 1 - p}_{\prim}(X) = \sum_{[Q] \in P/G} (-1)^{d - \dim Q} 
\Ind_{G_Q}^{G} [ \det(\rho_Q) \cdot \phi_{Q, p + 1} ], 
\]
where $P/G$ denotes the set of $G$-orbits of faces of $P$. 
\end{ntheorem}

\excise{
In particular, we deduce that $H^{d - 1,0}(X) = \det(\rho) \cdot \chi_P^*$, where $\chi_P^*$ denotes the permutation representation corresponding to the action of $G$ on the interior lattice points $\Int(P) \cap M$ of $P$ (Corollary~\ref{c:genus}). In particular, $\dim H^{d - 1,0}(X/G)$ equals the number of $G$-orbits of $\Int(P) \cap M$  whose isotropy subgroup is contained in $\det(\rho)^{-1}(1)$
(Remark~\ref{r:genus}).
}

Let us further assume that $P$ is a \define{simplex} i.e. $P$ has precisely $d + 1$ vertices $\{ v_0, \ldots, v_d \}$.  Let $\Pi$ denote the set of interior lattice points of the 
parallelogram spanned by the vertices 
of $P \times 1$ in $M \oplus \Z$. 
That is,
\[
\Pi = \{ w \in M \oplus \Z \mid w = \sum_{i} \alpha_i (v_i, 1) \textrm{ for some } 0 < \alpha_i < 1 \}.
\]
Let $u: M \oplus \Z \rightarrow \Z$ denote projection onto the second co-ordinate, and let
$\Pi_k = \{ w \in \Pi \mid u(w) = k \}$, with corresponding permutation representation $\chi_{\lef \Pi_k \rig}$.

\begin{ncor}[Corollary~\ref{c:simplex}]
If $P$ is a simplex, then for any $p \ge 0$,
\[
H^{p, d - 1 - p}_{\prim} (X) = \det(\rho) \cdot \chi_{\lef \Pi_{p + 1} \rig}.
\]
\end{ncor}

In particular,  
let $H^{d - 1}_{\prim}(X/G)
= \bigoplus_{p} H^{p, d - 1 - p}_{\prim} (X/G)$ 
denote the subspace 
 corresponding to  
$H^{d - 1}_{\prim} (X)^G$ under the 
isomorphism $H^*(X/G) \cong H^* (X)^G$.
The above corollary implies that
$\dim H^{p, d - 1 - p}_{\prim} (X/G)$ equals the number of $G$-orbits of $\Pi_{p + 1}$ whose isotropy subgroup is contained in $\det(\rho)^{-1}(1)$, and we deduce the Hodge structure of $X/G$ (Remark~\ref{r:quotient}). 

As a concrete example, consider the action of $\Sym_{d + 1}$ on the Fermat hypersurface $X_m = \{ x_0^m + \cdots + x_d^m = 0 \} \subseteq \P^d$ of degree $m$ by permuting co-ordinates (Example~\ref{e:Fermat2}). If $\sgn$ denotes the sign representation of $\Sym_{d + 1}$, then we deduce that 
 $\sgn \cdot H^{p, d - 1 - p}_{\prim}(X_m)$ is isomorphic to 
the permutation representation of $\Sym_{d + 1}$ on the set 
\begin{equation*}
\{ (a_0, \ldots, a_d) \in \Z^{d + 1} \mid 0 < a_i < m, \sum_{i = 0}^d a_i = (p + 1)m \}. 
\end{equation*}
An inexplicit formula for the characters of these representations 
can be deduced from general results of Ch\^enevert on actions of groups on smooth hypersurfaces in projective space \cite[Theorem~2.2]{CheRepresentations}. On the other hand, we deduce
that  if 
$g$ in $\Sym_{d + 1}$ has cycle type $( \lambda_1, \ldots, \lambda_r )$, then
$\tr(g; H^{p, d - 1 - p}_{\prim}(X_m))$ is equal to 
\[
(-1)^{d + 1 - r} 
\# \{ (a_1, \ldots, a_r) \in \Z^{r} \mid 0 < a_i < m, \sum_{i = 0}^d \lambda_i a_i = (p + 1)m \}.
\]

It follows that $\sgn \cdot H^{d - 1}_{\prim}(X_m)$  is isomorphic to the permutation representation of $\Sym_{d + 1}$ on the set 
\begin{equation*}
\{ (a_0, \ldots, a_d) \in (\Z/m\Z)^{d + 1} \mid a_i \ne 0, \sum_{i = 0}^d a_i = 0 \}. 
\end{equation*}
By standard comparison theorems (see, for example, Section~1 in \cite{KLEquivariant}), this agrees with the representation of $\Sym_{d + 1}$ on the primitive $l$-adic cohomology of $X_m$. In this case, 
a highly non-trivial proof of the latter result is due to Br\"unjes, who uses it to describe the zeta functions of all `twisted Fermat equations'  \cite[Corollary~11.3]{BruForms}.

Lastly, in Section~\ref{s:mirror}, we conjecture an equivariant version of Batyrev and Borisov's mirror symmetry between 
pairs of Calabi-Yau hypersurfaces in dual Fano toric varieties \cite{BBMirror}. 
If $P$ and $P^*$ are polar, $G$-invariant, \define{reflexive} polytopes, and $X$ and $X^*$ are corresponding $G$-invariant, non-degenerate hypersurfaces, then we introduce 
\define{equivariant stringy invariants}
$E_{\st,G}(X; u,v)$ and $E_{\st, G}(X^*; u,v)$, which satisfy the property that if $\tilde{X} \rightarrow X$ is a $G$-invariant, crepant resolution, then $E_{\st,G}(X) = E_{G}(\tilde{X})$.

\begin{nconj}[Conjecture~\ref{c:mirror}]
The equivariant stringy invariants
$E_{\st,G}(X; u,v)$ and $E_{\st, G}(X^*; u,v)$ 
 are rational functions in $u$ and $v$ satisfying 
\[
E_{\st, G}(X; u,v) = (-u)^{d - 1} \det(\rho) \cdot E_{\st, G}(X^*; u^{-1},v). 
\]
\end{nconj}
If there exist $G$-equivariant, crepant resolutions
 $\tilde{X} \rightarrow X$ and $\tilde{X}^* \rightarrow X^*$, then the conjecture says that 
 \[
H^{p,q}(\tilde{X}) = \det(\rho) \cdot H^{d - 1 - p,q}(\tilde{X}^*) \in R(G) \: \textrm{  for  } \: 0 \le p,q \le d - 1. 
\]
This would have the surprising consequence that if 
$H = \det(\rho)^{-1}(1)$, then 
the (possibly singular) Calabi-Yau varieties  
$\tilde{X}/H$ and $\tilde{X}^*/H$ have mirror Hodge diamonds (Remark~\ref{r:surprise}). 



\begin{ncor}[Corollary~\ref{c:smoothmirror}, Corollary~\ref{c:dim3}]  
The conjecture holds in the following cases 
\begin{itemize}
\item if $X$ and $X^*$ are smooth, 

\item if $X$ and $X^*$ admit $G$-equivariant, crepant, toric resolutions and $\dim X \le 3$. 

\end{itemize}
\end{ncor}

We finish with an explicit example of equivariant mirror symmetry. 
Consider the action of $\Sym_{5}$ on the quintic $3$-fold $X = \{ x_0^5 + \cdots + x_d^5 = 0 \} \subseteq \P^4$ by permuting co-ordinates. 
Let $H$ be the quotient of the finite group $\{ (\alpha_0, \ldots, \alpha_4) \in  (\Z/5\Z)^{5} \mid \sum_{i = 0}^4 \alpha_i = 0 \}$ by the diagonally embedding subgroup $\Z/5\Z$. Then $(\alpha_0, \ldots, \alpha_4) \in H$ acts on 
$\P^4$ by multiplying co-ordinates by 
$(e^{\frac{2 \pi i \alpha_0}{5}}, \ldots, e^{\frac{2 \pi i \alpha_4}{5}})$. 
 The hypersurface
$Z_\psi = \{ x_0^{5} + 
 x_1^{5} +  x_2^{5} +  x_3^{5}  + x_4^{5} = \psi x_0x_1x_2x_3x_4 \} \subseteq \P^4$ is $H$-invariant and $\Sym_5$-invariant. If we set 
 $X^* = Z_\psi /H$ for a general choice of $\psi$, then $X^*$ inherits a $\Sym_5$-action, and $X^*$ may be regarded as a mirror to $X$. Moreover, there exists a $\Sym_5$-equivariant, crepant, toric resolution 
$\tilde{X}^* \rightarrow X^*$. Using the explicit calculations for Fermat hypersurfaces above, together with the above corollary, we deduce that if 
$\mu$ is the $101$-dimensional representation $1 + 2 \Ind_{\Sym_3}^{\Sym_5} (1)  + 2 \Ind_{\Sym_2 \times \Sym_2}^{\Sym_5}(1)$, then the representations of $\Sym_5$ on the cohomology of $X$ and $\tilde{X}^*$ are described in the Figure~\ref{reflexive}. 

\begin{figure}[htb]
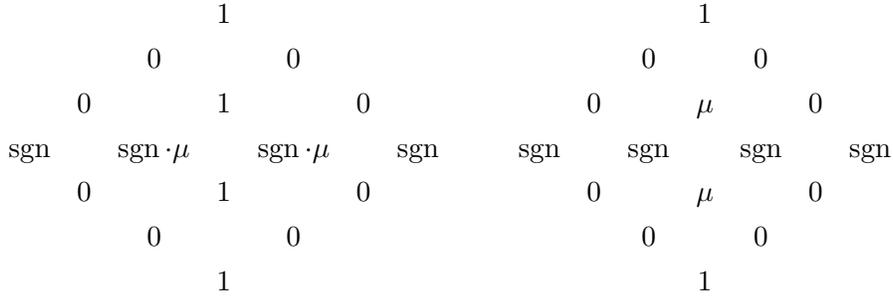


\begin{tabular}{ c    c     c   c  c  c  c                 c c  c                               c    c     c   c  c  c  c         }
 &  &  & $1$ & &                                             & &&      &        &  &  & $1$ & &   \\
  &  & $0$ & & $0$ &  &                                & &&           &  & $0$ & & $0$ &  &\\
   &  $0$  &  & $1$ & & $0$ &                    &  &&         &  $0$  &  & $\mu$ & & $0$ &\\
     $\sgn$  &   & $\sgn \cdot \mu$  &  & $\sgn \cdot \mu$ &  & $\sgn$      & &&        $\sgn$  &   & $\sgn$  &  & $\sgn$ &  & $\sgn$ \\
      &  $0$  &  & $1$ & & $0$ &              &&   &       &  $0$  &  & $\mu$ & & $0$ &  \\
        &  & $0$ & & $0$ &  &                         &&   &          &  & $0$ & & $0$ &  &\\
        &  &  & $1$ & &                               &          &   &&            &  &  & $1$ & & \\
\end{tabular}

 \caption{Equivariant Hodge diamonds for the quintic $3$-fold $X$ and an equivariant crepant resolution $\tilde{X}^*$ of its mirror.    }

        \label{reflexive}
        \end{figure}

If we restrict to the action of the subgroup $A_5 = \sgn^{-1}(1)$ of  $\Sym_5$ consisting of all even permutations, then we deduce that the Calabi-Yau varieties $X/A_5$ and $\tilde{X}^*/A_5$ have mirror Hodge diamonds, computed in Figure~\ref{reflexive2} below. 

\begin{figure}[htb]
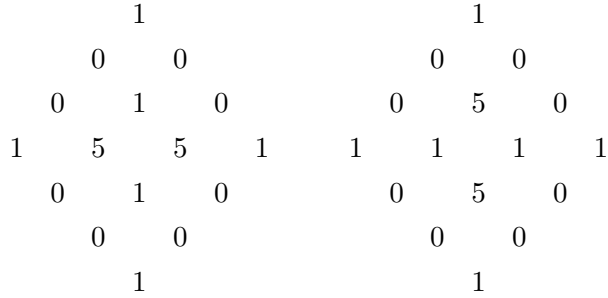


\begin{tabular}{ c    c     c   c  c  c  c                 c c  c                               c    c     c   c  c  c  c         }
 &  &  & $1$ & &                                             & &&      &        &  &  & $1$ & &   \\
  &  & $0$ & & $0$ &  &                                & &&           &  & $0$ & & $0$ &  &\\
   &  $0$  &  & $1$ & & $0$ &                    &  &&         &  $0$  &  & $5$ & & $0$ &\\
     $1$  &   & $5$  &  & $5$ &  & $1$      & &&        $1$  &   & $1$  &  & $1$ &  & $1$ \\
      &  $0$  &  & $1$ & & $0$ &              &&   &       &  $0$  &  & $5$ & & $0$ &  \\
        &  & $0$ & & $0$ &  &                         &&   &          &  & $0$ & & $0$ &  &\\
        &  &  & $1$ & &                               &          &   &&            &  &  & $1$ & & \\
\end{tabular}

 \caption{Hodge diamonds for $X/A_5$ and $\tilde{X}^*/A_5$.}
        \label{reflexive2}
        \end{figure}





We end the introduction with a brief outline of the contents of the paper.  In Section~\ref{s2} we 
provide the setup for the rest of the paper. In Section~\ref{s:ee} we recall some results about equivariant Ehrhart theory proved in \cite{YoEquivariant}. 
In Section~\ref{toric} we recall some basic facts about toric geometry and non-degenerate hypersurfaces. In Section~\ref{s:HodgeDeligne} we introduce equivariant Hodge-Deligne polynomials and provide some basic properties and examples. In Section~\ref{s:ehd} we prove our algorithm for computing the equivariant Hodge-Deligne polynomial of a non-degenerate hypersurface in a torus, and give several consequences. In Section~\ref{s:cohomology} and Section~\ref{s:simplex} we restrict to the cases when $P$ is simple and $P$ is a simplex respectively. In Section~\ref{s:mirror} we prove our results on equivariant mirror symmetry. 
We claim no originality when $G$ is trivial. In this case, our technique reduces to a slight variant of Danilov and Khovanskii's work in \cite{DKAlgorithm}, and our results are known.

\medskip

\noindent
{\it Notation and conventions.}  
All representations and cohomology groups will be defined over $\C$, unless otherwise stated.
All representations are finite-dimensional. 
 We often identify a representation $\chi$ with its associated character and write $\chi(g)$ for the evaluation of the character of $\chi$ at $g \in G$. We consider  representations of $G$ in the representation ring $R(G)$, and write $\chi + \phi$ (respectively $\chi \cdot \phi$) for the direct sum 
(respectively tensor product) of two representations $\chi$ and $\phi$.  We write $1 \in R(G)$ to denote the trivial representation. 
If $G$ acts on a set $S$, then we write $\chi_{\lef S \rig}$ for the corresponding permutation representation. 

\medskip
\noindent
{\it Acknowledgements.}  
The author would like to thank Kalle Karu, Gus Lehrer, John Stembridge and Jonathan Wise
for several useful discussions. The author benefited greatly from attending Denis Auroux's inspiring course on mirror symmetry at UC Berkeley in Fall 2009.  

\section{The setup}\label{s2} 

In this section, we introduce and justify the setup we will use throughout the paper.

Let $G$ be a finite group acting linearly on a lattice $M' \cong \Z^n$, and let $P$ be a $d$-dimensional $G$-invariant lattice polytope.
Observe that the affine span $W$ of $P$ in $M'_\R$ is $G$-invariant. If we fix a lattice point $\bar{u} \in W \cap M'$, then $M := W \cap M' - \bar{u}$ has the structure of a lattice of rank $d$ and $G$ acts linearly on $M$ via
\[
g \cdot (u - \bar{u})  = gu - g\bar{u} = (gu  - g\bar{u} + \bar{u}) - \bar{u},
\] 
for all $g \in G$ and $u \in W \cap M'$. Regarding $P$ as a lattice polytope in $M$, we see that $P$ is 
invariant under $G$ `up to translation'. That is, if we set consider the function $w: G \rightarrow M$ defined by $w(g) = g\bar{u} - \bar{u}$, then $w(1) = 0$,  $w(gh) = w(g) + g \cdot w(h)$, and if we identify $P$ with the lattice polytope $P - \bar{u}$ in $M$, then  $g \cdot P = P - w(g)$ in $M$ for all $g \in G$. 


Conversely, assume that $G$ acts linearly on a $d$-dimensional lattice $M$, and $P$ is a $d$-dimensional lattice polytope which is invariant under $G$ `up to translation'.  That is, assume there exists a function $w : G \rightarrow M$ satisfying $w(1) = 0$ and $w(gh) = w(g) + g \cdot w(h)$, and such that $g \cdot P = P - w(g)$ for all $g \in G$. Then  $G$ acts linearly on the lattice $M' = M \oplus \Z$ as follows:  $g \cdot (u, \lambda) = (g \cdot u - \lambda w(g), \lambda)$ for any $g \in G$ and $(u,\lambda) \in M'$. If we identify $P$ with the lattice polytope $P \times  1$ in $M'$, then $P$ is invariant under the action of $G$.  
Note that we recover the original linear action of $G$ on $M$ and the induced action on $P$ `up to translation' via the action of $G$ on $M \times  0  \subseteq M'$ and $P \times  0$ respectively. 

The preceding discussion motivates the following \define{setup}:


\emph{Let $G$ be a finite group acting linearly on a lattice $M' = M \oplus \Z$ of rank $d + 1$ such that the projection $M' \rightarrow \Z$ is equivariant with respect to the trivial action of $G$ on $\Z$. Let $P \subseteq M_\R \times  1$ be a $G$-invariant, $d$-dimensional lattice polytope.  
}

\emph{
By identifying $M$ with $M \times 0$, we regard $M$ as a lattice with a linear $G$-action
$\rho: G \rightarrow GL(M)$. 
We let $\det(\rho)$  denote the linear character $\wedge^d \rho: G \rightarrow \{ \pm 1 \}$.  
We often identify $P$ with the lattice polytope $\{ u \in M_\R \mid u \times 1 \in P \}$ in $M_\R$, which is $G$-invariant `up to translation'. } 

\section{Equivariant Ehrhart theory}\label{s:ee}

In this section, we recall some results from \cite{YoEquivariant} on a representation-theoretic generalization of Ehrhart theory. We also record some useful representation theory lemmas. We continue with the notation of Section~\ref{s2}, and if $G$ acts on a set $S$, then we write $\chi_{\lef S \rig}$ for the corresponding permutation representation.

For any positive integer $m$, let $\chi_{mP} = \chi_{\lef mP \cap M \rig}$ denote the permutation representation corresponding to the action of $G$ on the lattice points $mP \cap M$ of $mP$, 
and let $\chi_{mP} = 1$ when $m = 0$. If $G$ acts on $M$ via $\rho: G \rightarrow GL(M)$, and 
$R(G)$ denotes the representation ring of $G$,  
then we may 
write
\begin{equation*}\label{e:hstar}
\sum_{m \ge 0} \chi_{mP} t^m = \frac{ \phi[t]}
{(1- t)\det(I - \rho t)}, 
\end{equation*}
for some power series  $\phi[t] = \phi_{P,G}[t] = \sum_{i \ge 0} \phi_i t^i \in R(G)[[t]]$, where
\[
\det(I - \rho t) = 1 - \rho \, t + \wedge^2  \rho \,  t^2 - \cdots + (-1)^d \wedge^d \rho \,  t^d.  
\]
The following well-known lemma is useful for interpreting this definition of $\phi[t]$. 


\begin{lemma}\label{l:exterior}
Let $G$ be a finite group and let $V$ be an $r$-dimensional representation. Then 
\begin{equation*}
\sum_{m \ge 0} \Sym^m V t^m =  \frac{ 1}{1 - Vt + \wedge^2 V t^2 - \cdots + (-1)^r \wedge^r V t^r}.
\end{equation*}
Moreover, if an element $g \in G$ acts on $V$ via a matrix $A$, and if $I$ denotes the identity $r \times r$ matrix, then both sides equal $\frac{1}{\det(I - tA)}$ 
when the associated characters are evaluated at $g$. 
\end{lemma}

The power series $h^*(t) = \sum_{i \ge 0} \dim \phi_i t^i$ is a polynomial of degree at most $d$, called the 
\emph{$h^*$-polynomial} of $P$ (see, for example, \cite{BRComputing}). In particular, if the virtual representations $\phi_i$ are effective representations, then $\phi[t]$ is a polynomial of degree at most $d$.  
For any positive integer $m$, let $\chi^*_{mP} = \chi_{\lef \Int(mP) \cap M \rig}$
denote the permutation representation corresponding to the action of $G$ on the interior lattice points
$\Int(mP) \cap M$ of $mP$. 

\begin{corollary}\cite[Corollary~6.6]{YoEquivariant}\label{c:reciprocity}
With the notation above, if $\phi[t]$ is a polynomial, then 
\begin{equation*}
\sum_{m \ge 1} \chi^*_{mP} t^m = \frac{ t^{d + 1}\phi[t^{-1}]}
{(1- t)\det(I - \rho t) }.
\end{equation*}
In particular, $\phi[t]$ has degree at most $d$ and 
$\phi_d = \chi^*_P$. 
\end{corollary}

We have an explicit description of $\phi[t]$ when $P$ is a simplex. 
 Recall that $P$ is a \emph{simplex} if it has precisely $d + 1$ vertices $\{ v_0, \ldots, v_d \}$ in $M$. In this case, 
  we define 
 \[
 \BOX(P) = \{ v \in M \oplus \Z  \mid  v = \sum_{i = 0}^d a_i (v_i, 1) \textrm{ for some } 0 \le a_i < 1 \}, 
 \]
 and 
let $u: M \oplus \Z \rightarrow \Z$ denote projection onto the second co-ordinate. 

\begin{proposition}\cite[Proposition~6.1]{YoEquivariant}\label{p:simplex}
With the notation above, if $P$ is a simplex, then $\phi_i$ is the permutation representation induced by the action of $G$ on $\{ v \in  \BOX(P)  \mid u(v) = i \}$. 
\end{proposition}

A $d$-dimensional lattice polytope $P$ in $M$ is \define{reflexive} if the origin is the unique interior lattice point of $P$ and every non-zero lattice point in $M$ lies on the boundary of $mP$ for some positive integer $m$. 


\begin{corollary}\cite[Corollary~6.9]{YoEquivariant}\label{c:reflexive}
If $P$ is a $G$-invariant lattice polytope and $\phi[t]$ is a polynomial, then 
$P$ is a translate of a reflexive polytope
if and only if $\phi[t] = t^d \phi[t^{-1}]$. 
\end{corollary}

 We say that a reflexive polytope $P$ is \define{non-singular} if the vertices of each facet of $P$ form a basis for $M$. If $P$ is a $G$-invariant, non-singular, reflexive polytope, then the fan $\triangle$ over the faces of $P$ determines a smooth, projective toric variety $Z = Z(\triangle)$, with an action of $G$ via toric morphisms.

\begin{proposition}\cite[Proposition~8.1]{YoEquivariant}\label{p:reflexive}
With the notation above, if $P$ is a $G$-invariant, non-singular, reflexive polytope, then
$\phi_i = H^{2i} Z \in R(G)$. 
\end{proposition}

We will often use the following lemma on permutation representations. 
If $G$ acts transitively on a set $S$, then the associated \emph{isotropy group} $H$ is 
the subgroup of $G$ which fixes a given $s$ in $S$, and is well-defined up to conjugation.

\begin{lemma}\label{l:permutation}
If  $G$ acts on a set $S$, then $\chi_{\lef S \rig}(g)$ equals the number of elements of $S$ fixed by $g$.
If $\lambda: G \rightarrow \C$ is a $1$-dimensional representation, then  
the multiplicity of $\lambda$ in $\chi_{\lef S \rig}$ is equal to the number of $G$-orbits of $S$
whose isotropy subgroup is contained in the subgroup $\lambda^{-1}(1)$ of $G$.
\end{lemma}

We will also need the following simple lemma. 

\begin{lemma}\label{l:dual}
Suppose $G$ acts linearly on a lattice $N$ of rank $r$. Then we have isomorphisms of 
$G$-representations $\bigwedge^i N_\C \cdot \bigwedge^r N_\C \cong \bigwedge^{r - i} N_\C$.   
\end{lemma}
\begin{proof}
If an element $g \in G$ acts of $N_\C$, then, since $g$ has finite order, we may assume, after a change of basis, that $g$ acts via a diagonal matrix
 $(\lambda_1,\ldots, \lambda_r)$,  for some roots of unity $\lambda_i$. 
Since $\lambda_i^{-1} = \overline{\lambda}_i$ and  
$g$ acts on $\bigwedge^r N_\C$ via multiplication by $\pm 1$, it follows that
$(\lambda_1\cdots \lambda_r)^2 = 1$. 
We conclude that the left hand side evaluated at $g$ is equal to 
 \[
 \lambda_1\cdots \lambda_r \sum_{k_1 < \cdots < k_i} \lambda_{k_1}\cdots \lambda_{k_i} = 
  \sum_{k'_1 < \cdots < k'_{ r- i}} \overline{\lambda}_{k'_1} \cdots \overline{\lambda}_{k'_{r - i}}
  =   \sum_{k'_1 < \cdots < k'_{ r- i}} \lambda_{k'_1} \cdots \lambda_{k'_{r - i}}. 
   \]
\end{proof}

\section{Toric geometry and non-degenerate hypersurfaces}\label{toric}

In this section, we recall some basic facts about toric varieties and non-degenerate hypersurfaces in tori. We refer the reader to \cite{FulIntroduction} and \cite{TevCompactifications} for proofs of the statements below. 

We continue with the notation of Section~\ref{s2}. That is,
let $G$ be a finite group acting linearly on a lattice $M' = M \oplus \Z$ of rank $d + 1$ such that the projection $M' \rightarrow \Z$ is equivariant with respect to the trivial action of $G$ on $\Z$. Let $P \subseteq M_\R \times 1 $ be a $G$-invariant, $d$-dimensional lattice polytope.  In what follows, we often consider $P$ as a lattice polytope in $M_\R$.

If we let $\sigma$ denote the cone over $P \times 1$ in $M_\R'$, then $G$ acts on the $\N$-graded, semi-group algebra $R = \C[\sigma \cap M']$. This induces an action of $G$ 
on the projective toric variety  $Y = \Proj R$ with torus $T = \Spec \C[M]$ via toric morphisms. If $N = \Hom(M, \Z)$ is the dual lattice to $M$, then $Y$ is the toric variety determined by the \emph{normal fan} to $P$ in $N_\R$, 
and comes equipped with a 
$T$-equivariant ample line bundle $L$, which is preserved under the action of $G$. 
We may identify the action of $G$ on $H^0(Y, L^{\otimes m})$ with the action of $G$ on the  $m^{\textrm{th}}$ graded piece $R_m$ of $R$, and hence with the permutation representation $\chi_{mP}$ induced by the action of $G$ on $mP \cap M'$.




If $u \in M$ corresponds to the monomial $\chi^u \in \C[M]$, then a hypersurface $X^\circ = \{ \sum_{ u \in P \cap M} a_u \chi^u = 0 \} \subseteq T$ defines a $G$-invariant hypersurface of $T$ if and only if $a_u = a_u' \in \C$ whenever $u$ and $u'$ lie in the same $G$-orbit of $P \cap M$.  
The closure $X$ of $X^\circ$ in $Y$
 is $G$-invariant and  may be regarded as the zero locus of a 
section of $L$. 
The \define{Newton polytope} of $X^{\circ}$ is the convex hull of $\{ u \in M \mid a_u \ne 0 \}$ in $M_\R$.

We will need the notion of a \define{non-degenerate} hypersurface in a torus. Non-degenerate hypersurfaces were first studied by Khovanski{\u\i} \cite{KhoNewton}, 
and,
recently, have been extended to the notion of a \emph{Sch\"on} subvariety of a torus by Televev 
 \cite{TevCompactifications}.
Recall that if $\P(\triangle)$ is a complete toric variety corresponding to a fan $\triangle$ in a lattice $N$, then each cone $\tau$ in $\triangle$ corresponds to a torus orbit $T_\tau = \Spec \C[M_\tau]$ in $\P(\triangle)$, where $M_\tau$ denotes the intersection of $M = \Hom(N,\Z)$ with 
$\tau^{\perp} = \{ u \in M_\R \mid \; \lef u, v \rig = 0 \; \textrm{ for all } v \in \tau \}$. If $\triangle$ is the normal fan to $P$, and $\tau_Q$ is the cone in $\triangle$ corresponding to a face $Q$ of $P$, then we will write $T_Q = T_{\tau_Q}$.


\begin{definition}\label{d:nondegenerate}
With the notation above, 
let  $Z^\circ \subseteq T = \Spec \C[M]$ be a  hypersurface, and let $Z$ denote the closure of $Z^\circ$ in $\P(\triangle)$. 
Then $Z^\circ$
 is \define{non-degenerate} with respect to $\P(\triangle)$ if the intersection $Z \cap T_\tau$ of $Z$ with each torus orbit $\{ T_\tau \mid \tau \in \triangle \}$ is a smooth (possibly empty) hypersurface in $T_\tau$. 

The hypersurface 
$Z^\circ$ is \define{non-degenerate with respect to P} if 
$Z^\circ$ is non-degenerate with respect to the projective toric variety $Y$ corresponding to $P$, and $P$ is  the Newton polytope of $Z^\circ$. 
\end{definition}

\begin{remark}\label{r:induction}
 One can show that a hypersurface $Z^\circ = \{ \sum_{ u \in P \cap M} a_u \chi^u = 0 \} \subseteq T$ is non-degenerate with respect to $P$ if and only if $\{ \sum_{ u \in Q \cap M} a_u \chi^u = 0 \}$ defines a smooth (possibly empty) hypersurface in $T$ for each face $Q$ of $P$. 
 Moreover, in this case, $Z \cap T_Q$ is non-degenerate with respect to $Q$. 
\end{remark}



If  $Z^\circ \subseteq T$
 is non-degenerate with respect to $\P(\triangle)$, then the completion of the local ring of $Z$ at $z$ is isomorphic to the completion of the local ring of $\P(\triangle)$ at $z$. 
 In particular, 
 $Z$ is smooth if and only if $\P(\triangle)$ is smooth away from its torus fixed points. 
 Moreover, if $\P(\Sigma) \rightarrow P(\triangle)$ is  a proper, birational toric morphism, then 
 $Z^\circ$ is non-degenerate with respect to $\P(\Sigma)$, and the closure $Z'$ of $Z^\circ$ in $\P(\Sigma)$ is 
 the inverse image of $Z$. 

In our case, assume that $X^\circ \subseteq T$ defines a $G$-invariant, non-degenerate hypersurface with respect to $P$. There exists a smooth, complete 
toric variety $\P = \P(\Sigma)$ with a $G$-action via toric morphisms, and a  $G$-invariant, proper, birational morphism $f: \P \rightarrow Y = Y(\triangle)$, where $\triangle$ is the normal fan to $P$ \cite{AWEquivariant}. If $X'$ denotes the closure of $X^\circ$ in $\P$, then, by the above discussion, we obtain a  $G$-equivariant resolution of singularities $X' \rightarrow X$.

For every cone $\tau'$ in $\Sigma$, let $\tau = \tau_Q$ denote the smallest cone in 
$\triangle$ containing $\tau'$, for some face $Q$ of $P$. If $G_Q$ denotes the isotropy group of $Q$ (i.e. the subgroup of $G$ which leaves $Q \subseteq P$ invariant), then $G_Q$ acts on the lattice $M_{\tau'}/M_{\tau}$, and hence on the corresponding torus $T_{\tau',f} := \Spec \C[M_{\tau'}/M_{\tau}]$. Moreover,  $f$ induces a $G_Q$-equivariant projection 
 \begin{equation}\label{e:resolution}
 X' \cap T_{\tau'} \isom (X \cap T_{\tau_Q}) \times T_{\tau',f}  \rightarrow X \cap T_{\tau_Q}. 
 \end{equation}

We may regard $X'$ as a section of the (globally generated) line bundle $f^*L$ on $\P$. For any non-negative integer $m$, we have  isomorphisms of $G$-representations
\begin{equation}\label{e:vanishing}
H^i(\P, \mathcal{O}_\P(mX')) \cong  \left\{\begin{array}{cl} 
H^0(Y, L^{\otimes m}) = \chi_{mP} & \text{if } i = 0 \\ 0 & \text{ otherwise. } \end{array}\right. 
\end{equation}

Finally, we recall some basic Hodge theory (cf. Section~\ref{s:HodgeDeligne}). Let $Z$ be a smooth, complete $n$-dimensional variety, let $D$ be a simple normal crossings divisor, and set $Z^\circ = Z \smallsetminus D$. 
The sheaf $\Omega^1_{Z }(\log D)$ of \define{rational forms on $X$ with log poles on $D$}  is locally described as follows: 
if $z_1, \dots, z_d$ are local co-ordinates of $Z$ and $D$ is locally defined by $z_1z_2\cdots z_r = 0$, then $\Omega_{Z}^1(\log D)$ is the free $\mathcal{O}_Z$-module degenerated by $\frac{dz_1}{z_1}, \ldots, \frac{dz_r}{z_r}, dz_{r + 1}, \ldots, dz_n$. For any positive integer $p$, $\Omega^p_{Z }(\log D) = 
\bigwedge^p \Omega^1_{Z }(\log D)$, and $\Omega^p_{Z }(\log D) = \mathcal{O}_Z$ when $p = 0$. 
If $G$ acts algebraically on $Z$ and leaves $D$ invariant, then we obtain an isomorphism of $G$-representations:
\begin{equation}\label{e:Hodge}
F^p H^k Z^\circ/ F^{p + 1} H^k Z^\circ   \cong H^{k - p}(Z, \Omega^{p}_{Z}(\log D) ).
\end{equation}

\section{Equivariant Hodge-Deligne polynomials}\label{s:HodgeDeligne}

In this section, we introduce the equivariant Hodge-Deligne polynomial of a complex variety with group action. This is a slight generalization of the notion of weight polynomial considered by Dimca and Lehrer in \cite{DLPurity}, and  the notion of   \emph{equivariant $\chi_y$-genus} considered by Cappell, Maxim and Shaneson in \cite{CMSEquivariant}.


Let $G$ be a finite group acting algebraically on a  $d$-dimensional complex variety $Z$. A famous result of Deligne states that the cohomology of $Z$ carries a \emph{mixed Hodge structure}. In particular, 
the $k^{\textrm{th}}$ cohomology group $H^k_c Z = H^k_c (Z; \C)$ of $Z$ with compact support has an increasing \emph{weight filtration}
\[
0 \subseteq W_{0} H^k_c Z \subseteq W_1 H^k_c Z \subseteq \cdots \subseteq W_{k} H^k_c  Z = H^k_c Z
\]
and a decreasing Hodge filtration 
\[
H^k_c Z = F^0 H^k_c Z \supseteq \cdots \supseteq F^{d}  H^k_c Z \supseteq 0
\]
which induces a pure Hodge structure of weight $m$ on 
\[
Gr_m^W  H^k_c Z = W_{m} H^k_c Z / W_{m  - 1} H^k_c Z.
\]
 The action of $G$ preserves the mixed Hodge structure and hence we have 
 induced $G$-representations 
 on 
$H^{p,q}( H^k_c Z)$,  the $(p,q)^{\textrm{th}}$ piece of $Gr_{p + q}^W  H^k_c Z$, for $p + q \le k$. If $R(G)$ denotes the representation ring of $G$, then we may consider the (virtual) representation 
\[
e^{p,q}_G (Z)  := \sum_{k = 0}^{p + q} (-1)^k H^{p,q}( H^k_c Z) \in R(G). 
\]
\begin{definition}
If a finite group $G$ acts algebraically on a complex variety $Z$, then 
the \define{equivariant Hodge-Deligne polynomial} is 
\[
E_G (Z) = E_G(Z;u,v) = \sum_{p,q} e^{p,q}_G (Z) u^p v^q \in R(G)[u,v]. 
\]
\end{definition}

\begin{remark}\label{r:symmetry}
Since the action of $G$ on  $H^k_c Z = H^k_c (Z; \C)$ is induced by the action of $G$ on $H^k_c (Z; \Z)$,  it follows that complex conjugation  commutes with the $G$-action on $H^k_c (Z; \C)$. Hence we have an isomorphism of $G$-representations $H^{p,q}( H^k_c Z) \cong \overline{H^{p,q}}( H^k_c Z) = 
H^{q,p}( H^k_c Z)$.  In particular, $e^{p,q}_G (Z) = e^{q,p}_G (Z)$, and $E_G (Z)$ is symmetric in $u$ and $v$.  
\end{remark}


If $U$ is a $G$-invariant open subset of $Z$ and $V = X \setminus U$, then the long exact sequence of cohomology with compact support
\[
\cdots  \rightarrow  H^{k - 1}_c V \rightarrow H^k_c U \rightarrow H^k_c X \rightarrow H^k_c V \rightarrow H^{k + 1}_c U \rightarrow \cdots 
\]
consists of morphisms of mixed Hodge structures. In particular, it follows that the equivariant Hodge-Deligne polynomial satisfies the following \define{additivity property}: 
\[
E_G(Z)  = E_G(U)  + E_G(V)  \in R(G)[u,v]. 
\]
If $G$ acts algebraically on varieties $V$ and $V'$, then $G$ acts algebraically on $V \times V'$, and, since the K\"unneth isomorphism respects mixed Hodge structures and the action of $G$, the equivariant Hodge-Deligne polynomial satisfies the following \define{multiplicative property}: 
\[
E_G(V \times V')  = E_G(V)E_G(V')  \in R(G)[u,v]. 
\]

\begin{example}\label{e:smooth}
If $Z$ is a complete variety of dimension $r$ with at worst quotient singularities, then $H^k_c Z = H^k Z$ admits a pure Hodge structure of weight $k$ i.e. $W_{k - 1} H^k_c Z  = 0$. In this case, $E_G(Z) = \sum_{p,q} (-1)^{p + q} H^{p,q}(Z)  u^p v^q$ encodes the representations of $G$ on the $(p,q)$-pieces of the cohomology of $Z$. Moreover, Poincar\'e duality induces an isomorphism of $G$-representations $H^{p,q}(Z) \cong H^{r - p, r - q}(Z)$, and hence $E_G(Z; u,v) = (uv)^r   E_G(Z; u^{-1},v^{-1})$ \cite{FujDuality} (cf. \cite[1.6]{DLPurity}). If $Z$ is projective, then successive capping with a hyperplane class gives an explicit isomorphism of $G$-representations $H^{p,q}(Z) \cong H^{r - q, r - p}(Z)$ \cite[p. 64]{StaNumber}. 
\end{example}

\begin{example}\label{e:tori}
If $G$ acts linearly on a lattice $M$ of rank $d$ via $\rho: G \rightarrow GL(M)$, then $G$ acts algebraically on the corresponding torus 
$T = \Spec \C[M]$, and we have canonical isomorphisms of $G$-representations $H^{ d + k}_c T = H^{k,k}(H^{ d + k}_c T) \cong \wedge^{d - k} \rho$. In particular, $E_G(T) = \sum_{k = 0}^{d} (-1)^{d + k}  \bigwedge^{d - k} \rho \:  (uv)^k$ (cf. proof of Theorem 1.1 in \cite{LehRational}). 
\end{example}

If $H$ is a subgroup of $G$ acting on a variety $Z$, then we write $\Ind_{H}^G E_H(Z)  = \Ind_{H}^G E_H(Z; u,v) = \sum_{p,q} \Ind_H^G e_H^{p,q} u^p v^q$ for the polynomial of induced (virtual) representations in $R(G)[u,v]$.   We will need the following proposition.

\begin{proposition}\cite[Proposition 2.3]{LehRational}\label{p:induced}
Suppose a finite group $G$ acts a complex variety $Z$, and $Z$ admits a decomposition into locally closed subvarieties $Z = \coprod_{i \in I} Z_i$ which are permuted by $G$.  Then 
\[
E_G (Z) = \sum_{\iota \in I /G} \Ind_{G_i}^G E_{G_i}(Z_i),
\]
where $I/G$ denotes the set of orbits of $G$ acting on $I$, $i$ denotes a representative of the orbit $\iota$, and $G_i$ denotes the isotropy group of $i$ in $I$. 
In terms of characters, for any $g$ in $G$, 
\[
E_G (Z)(g) = \sum_{g \cdot Z_i = Z_i} E_{G_i}(Z_i)(g).
\]
\end{proposition} 

\begin{example}\label{e:toric}
A  toric variety $X = X(\triangle)$ corresponding to a fan $\triangle$ is a disjoint union of tori $\{ T_\tau \mid \tau \in \triangle \}$ (see Section~\ref{toric}). 
If a finite group $G$ acts on $X$ via toric morphisms, then $G$ permutes the tori $\{ T_\tau \mid \tau \in \triangle \}$, and hence one immediately deduces an expression for the equivariant Hodge-Deligne polynomial $E_G(X)$ from Example~\ref{e:tori} and Proposition~\ref{p:induced} (cf. Theorem 1.1 in \cite{LehRational}). 
\end{example}

\section{Equivariant Hodge-Deligne polynomials of hypersurfaces in tori}\label{s:ehd}

In this section, we present an algorithm to determine the equivariant Hodge-Deligne polynomial of a $G$-invariant, non-degenerate hypersurface $X^\circ$ in a torus. Equivalently, we determine the representations of $G$ on the pieces of the mixed Hodge structure on $H^k_c X^\circ$ (Remark~\ref{r:Hodge-Deligne}).  This result and its proof may be viewed as an equivariant analogue of Danilov and Khovanski{\u\i}'s work in  \cite{DKAlgorithm}.

We continue with the notation from Section~\ref{s2}. 
That is, $G$ is a finite group acting linearly on a lattice $M$ of rank $d$ via $\rho: G \rightarrow GL(M)$, and $P$ is a $d$-dimensional lattice polytope in $M$ which is $G$-invariant `up to translation'. 
Let $X^\circ \subseteq T = \Spec \C[M]$ be a $G$-invariant, non-degenerate hypersurface with Newton polytope $P$, and let $X$ denote the closure of $X^\circ$ in the projective toric variety $Y$ corresponding to the normal fan of $P$. 
\begin{lemma}\label{l:affine}
$H^k_c X^\circ = 0$ for $k < d - 1$. 
\end{lemma}
\begin{proof}
Since $X^\circ$ is a smooth, affine, $(d - 1)$-dimensional variety, this follows from a classical result of Andreotti and Frankel \cite[Theorem 1]{AFLefschetz}. 
\end{proof}

\subsection{Step 1}\label{step1}
We have the following Lefschetz type result due to Danilov and Khovanski{\u\i}.

\begin{proposition}\cite[Proposition 3.9]{DKAlgorithm}\label{p:Gysin}
The Gysin map $H^k_c X^\circ \rightarrow H^{k + 2}_c T$ is an isomorphism for $k > d - 1$, and a surjection for  $k = d - 1$. 
\end{proposition}

The isomorphism in the above lemma is a morphism of mixed Hodge structures of type $(1,1)$ which is equivariant with respect to $G$. 
Since $H^{p,q}(H^k_c X^\circ) = 0$ for $p + q > k$, we conclude, using Lemma~\ref{e:tori}, that if $p + q > d - 1$, then
\begin{equation}\label{e:Gysin}
e^{p,q}_G(X^\circ) = e^{p + 1, q + 1}_G(T) 
=  \left\{\begin{array}{cl} 
(-1)^{d - 1 - p} \bigwedge^{d - 1 - p} \rho & \text{if } p = q ; \\ 0 & \text{otherwise}. \end{array}\right. 
\end{equation}

Combined with Lemma~\ref{l:affine} and Example~\ref{e:tori}, we conclude that we understand the representations $H^k_c X^\circ$ for $k \neq d - 1$. Moreover, if we set
\[
H^{d - 1}_{c,\prim} X^\circ := \ker[H^{d - 1}_c X^\circ \rightarrow H^{d + 1}_c T],
\]
then $H^{d - 1}_{c,\prim} X^\circ$ inherits a mixed Hodge structure, compatible with the action of $G$, and we have an isomorphism of $G$-representations $H^{d - 1}_c X^\circ \cong H^{d - 1}_{c,\prim} X^\circ \oplus H^{d + 1}_c T$. Hence it remains to understand the action of $G$ on the mixed Hodge structure of $H^{d - 1}_{c,\prim} X^\circ$.  

\begin{remark}\label{r:Hodge-Deligne}
It follows from the above discussion that 
the equivariant Hodge-Deligne polynomial $E_G(X^\circ)$ determines the $G$-representations  $H^{p,q}( H^k_c X^\circ)$. 
\end{remark}

\subsection{Step 2}\label{step2}

With the notation of Section~\ref{toric}, let  $\P = \P(\Sigma)$  be a 
complete 
toric variety with at worst quotient singularities and with a  $G$-action via toric morphisms, admitting a $G$-invariant, proper, birational morphism $f: \P \rightarrow Y = Y(\triangle)$. If $X'$ denotes the 
 closure of $X^\circ$ in $\P$, then  $X'$ is $G$-invariant and has at worst quotient singularities. 
By Proposition~\ref{p:induced}, 
\[
E_G (X'; u,v) = \sum_{[\tau'] \in \Sigma /G} \Ind_{G_{\tau'}}^G E_{G_{\tau'}}(X' \cap T_{\tau'}),
\]
where  $\Sigma/G$ denotes the set of orbits of $G$ acting on the cones in $\Sigma$, $\tau'$ denotes a representative of an orbit, and $G_{\tau'}$ denotes the isotropy group of $\tau'$. 
For every cone $\tau'$ in $\Sigma$, let $\tau = \tau_Q$ denote the smallest cone in the normal fan $\triangle$ containing $\tau'$, for some face $Q$ of $P$, and write $f(\tau') = Q$. Since $G_{\tau'}$ is a subgroup of the isotropy group of $Q$ in $P$, it follows from \eqref{e:resolution} and the multiplicative property of equivariant Hodge-Deligne polynomials that 
\[
E_G (X'; u,v) = \sum_{[\tau'] \in \Sigma /G} \Ind_{G_{\tau'}}^G [E_{G_{\tau'}}(X \cap T_{f(\tau')})E_{G_{\tau'}}( T_{\tau',f})],
\]
where $T_{\tau',f} := \Spec \C[M_{\tau'}/M_{\tau}]$. We conclude, using Remark~\ref{r:induction} and  Example~\ref{e:tori}, that
\[
E_G (X'; u,v) = E_G(X^\circ;u,v) + \alpha(u,v),
\]
where $\alpha(u,v) \in R(G)[u,v]$ is known by induction on dimension. Since $X'$ is smooth and complete,  Example~\ref{e:smooth} implies that  
$E_G (X'; u,v) = (uv)^{d - 1} E_G (X'; u^{-1},v^{-1})$, and hence we know the difference 
$E_G(X^\circ;u,v) - (uv)^{d - 1}E_G(X^\circ;u^{-1},v^{-1})$.  By Step~$1$, we know $e^{p,q}_G(X^\circ)$ for $p + q > d - 1$, and hence we deduce $e^{p,q}_G(X^\circ)$ for $p + q < d - 1$. 

\subsection{Step 3}\label{step3}

It remains to determine $e^{p,q}_G(X^\circ)$ for $p + q = d - 1$. Clearly, it will be enough to compute the sums  $\sum_q e_G^{p,q}(X^\circ)$, or, equivalently, the polynomial $E_G(X^\circ; u , 1)$. 
Using the fact that Poincar\'e duality preserves the mixed Hodge structure \cite{FujDuality} (cf. \cite[1.6]{DLPurity}), we have
\begin{align*}
\sum_q e^{p,q}_G(X^\circ) &=  \sum_q \sum_k (-1)^k H^{p,q}(H^k_c X^\circ)  \\
&= 
\sum_q \sum_k (-1)^k H^{d - 1 - p, d - 1 - q}(H^{2d - 2 - k}X^\circ) \\
&= \sum_q \sum_k (-1)^k H^{d - 1 - p, q}(H^{k}X^\circ) \\
&= \sum_k (-1)^k F^{d - 1 - p} H^k X^\circ/ F^{d - p} H^k X^\circ.  
\end{align*}
We continue with the notation of Step 2, 
and  
let $X'$ denote the (smooth, $G$-invariant) compactification of $X^\circ$ in $\P = \P(\Sigma)$. 
Let $D = D_1 + \cdots + D_r$ denote the union of the $T$-invariant divisors of $\P$ and let
$D_{X'} = D_1 \cap X' + \cdots + D_r \cap X'$. Our assumption that $X^\circ$ is non-degenerate with respect to $P$ implies that $D$ and $D_{X'}$ are simple normal crossings divisors in $\P$ and $X'$ respectively. 
It follows from \eqref{e:Hodge} that we need to compute the virtual representation
\begin{equation}\label{e:stepA}
\sum_k (-1)^k F^{d - 1 - p} H^k X^\circ/ F^{d - p} H^k X^\circ = (-1)^{d - 1 - p} \chi(X', \Omega^{d - 1 - p}_{X' }(\log D_{X'}) ). 
\end{equation}

One verifies that we have exact sequences of $G$-equivariant sheaves
\[
0 \rightarrow \Omega^{\bullet - 1}_{X'}(\log D_{X'}) \otimes \mathcal{O}_{\P}(-X')|_{X'}  \rightarrow \Omega^{\bullet}_{\P}(\log D)|_{X'} \rightarrow \Omega^{\bullet}_{X'}(\log D_{X'})  \rightarrow 0  
\]


Taking Euler characteristics and twists by $\mathcal{O}_{\P}(kX')$ gives
\[
 \chi(X', \Omega^{d - 1 - p}_{X' }(\log D_{X'}) ) = \sum_{k = 0}^p (-1)^k \chi( X' , \Omega^{d - p + k}_{\P}(\log D) \otimes \mathcal{O}_{\P}( (k + 1) X')|_{X'}).
\]
From the exact sequence 
\[
0 \rightarrow \mathcal{O}_{\P}(-X') \rightarrow \mathcal{O}_{\P} \rightarrow \mathcal{O}_{X'} \rightarrow 0
\]
we obtain the following expression for $\chi(X', \Omega^{d - 1 - p}_{X' }(\log D_{X'}) ) $,
\[
\sum_{k = 0}^p  (-1)^k [ 
 \chi( \P , \Omega^{d - p + k}_{\P}(\log D) \otimes \mathcal{O}_{\P}( (k + 1) X')) - 
 \chi( \P , \Omega^{d - p + k}_{\P}(\log D) \otimes \mathcal{O}_{\P}( k X')) ].
\]
Rearranging gives
\[
-\chi(X', \Omega^{d - 1 - p}_{X' }(\log D_{X'}) )  = \chi(\P,  \Omega^{d - p}_{\P} (\log D)) \]
\[ + \sum_{k = 1}^{p + 1} (-1)^{k}
  [\chi(\P,  \Omega^{d - 1 - p + k}_{\P}(\log D) \otimes \mathcal{O}_{\P} ( k  X' )) + 
\chi(\P,  \Omega^{d - p + k}_{\P}(\log D) \otimes \mathcal{O}_{\P} ( k  X' ))  ].
\]
We need the following well-known lemma. Under the isomorphism below,  $u \in M$ corresponds to $d\chi^u/\chi^u \in \Omega_{\P}(\log D)$.  

\begin{lemma}\cite[Section~5]{BCOn}\label{l:logpoles}
There is a natural $G$-equivariant isomorphism 
\[
\Omega^{k}_{\P}(\log D) \cong \bigwedge^k M \otimes_\Z \mathcal{O}_{\P}. \] 
\end{lemma}

Recall that $\chi_{mP}$ denotes 
the permutation representation given by the action of $G$ on the lattice points $mP \cap M$. 
 It follows from Lemma~\ref{l:logpoles} and \eqref{e:vanishing} that, for any non-negative integer $m$,  
 \[
 \chi(\P,  \Omega^k_{\P}(\log D) \otimes \mathcal{O}_{\P}(mX')) =  \chi_{mP} \cdot \bigwedge^k \rho. 
 \]
 We obtain the following expression for $-\chi(X', \Omega^{d - 1 - p}_{X' }(\log D_{X'}) )$,
 \[
  \bigwedge^{d - p} \rho 
+ \sum_{k = 1}^{p + 1} (-1)^{k}
  [\: \chi_{kP} \cdot \bigwedge^{d - 1 - p + k} \rho \: + \:
\chi_{kP} \cdot \bigwedge^{d - p + k} \rho  \:  ].
\]

 By Lemma~\ref{l:dual}, 
 if we set $\rho' = \rho + 1$ and $\det(\rho) = \wedge^d \rho$, then 
\begin{equation}\label{e:stepB}
\chi(X', \Omega^{d - 1 - p}_{X' }(\log D_{X'}) ) = \bigwedge^{d - p - 1} \rho
- \det(\rho) \cdot \sum_{k = 0}^{p + 1} (-1)^{k} 
\chi_{kP}  \cdot  \bigwedge^{p + 1 - k} \rho'.
\end{equation}
Recall from Section~\ref{s:ee} that we consider a power series $\phi[t] = \sum_{i \ge 0} \phi_i t^i$ in $R(G)[[t]]$ of virtual representations defined by $\phi_0 = 1$ and 
\begin{equation}\label{e:stepC}
\phi_{p + 1} = (-1)^{p + 1} \sum_{k = 0}^{ p + 1} (-1)^k \chi_{kP} \cdot  \bigwedge^{p + 1 - k}  \rho'.  
\end{equation}
Putting together \eqref{e:stepA}, \eqref{e:stepB}, and \eqref{e:stepC}, yields our desired result. 
When $G$ is trivial, this follows from Equation~4.4 and Remark~4.6 in \cite{DKAlgorithm}.

\begin{theorem}\label{t:xycharacteristic}
With the notation above,
\[
\sum_q e^{p,q}_G(X^\circ) = (-1)^{d - 1 - p} \bigwedge^{d - 1 - p} \rho + (-1)^{d - 1}\det(\rho) \cdot \phi_{p + 1}. 
\]
\end{theorem}

As an immediate corollary, we see below that in this geometric situation the representations $\phi_i$ are effective representations. Given a finite group $G$ and $G$-invariant lattice polytope $P$, 
it is a very subtle question to determine when the virtual representations $\phi_i$ are effective representations (see Section~ 7 in \cite{YoEquivariant}). 

\begin{corollary}\label{c:vanishing}
If there exists a $G$-invariant, non-degenerate hypersurface $X^\circ \subseteq T$ with Newton polytope $P$, then 
$\phi_0$ is the trivial representation and 
\[
\phi_{p + 1} = \det(\rho) \cdot F^{p} H^{d - 1}_{c, \prim}  X^\circ/ F^{p + 1} H^{d - 1}_{c, \prim}  X^\circ
\]
for $p \ge 0$. In particular, the virtual representations $\phi_i$ are effective representations.
\end{corollary}
\begin{proof}
It follows from the definitions that $\phi_0$ is the trivial representation. 
By definition,  
\begin{equation}\label{e:xy}
\sum_q e^{p,q}_G(X^\circ) =  \sum_k (-1)^k F^{p} H^k_c X^\circ/ F^{p  + 1} H^k_c X^\circ.
\end{equation}
By Proposition~\ref{p:Gysin} and Example~\ref{e:tori}, 
\[
F^{p} H^{d - 1 + p}_c X^\circ/ F^{p + 1} H^{d - 1 + p}_c X^\circ \cong F^{p + 1} H^{d  + p + 1}_c T / F^{p + 2} H^{d +  p + 1}_c T \cong \bigwedge^{d - 1 - p} \rho. 
\]
Note that the above equality holds for $p > 0$, and, when $p = 0$, the equation holds if the first isomorphism is replaced by a surjection. 
Moreover, by Lemma~\ref{l:affine} and Proposition~\ref{p:Gysin}, the only other contribution to the right hand side of \eqref{e:xy} is 
$F^{p} H^{d - 1}_{c, \prim}  X^\circ/ F^{p + 1} H^{d - 1}_{c, \prim}  X^\circ$. The result now follows immediately from Theorem~\ref{t:xycharacteristic} using the fact 
that $\det(\rho)^2$ is the trivial representation. 
\end{proof}



We have the following immediate corollary. In the case when $G$ is trivial, this follows from 
Proposition~5.8 in \cite{DKAlgorithm}. 
Recall that $\chi^*_P$ denotes the permutation representation $\chi_{\lef \Int(P) \cap M \rig}$. 

\begin{corollary}\label{c:zero}
With the notation above,
\[
H^{d - 1, 0}(H^{d - 1}_c X^\circ) = (-1)^{d - 1} e^{d - 1,0}_G(X^\circ) = \det(\rho) \cdot \chi^*_P.
\]
\end{corollary}
\begin{proof}
By definition $e^{d - 1,q}_G(X^\circ) = \sum_{k \ge d - 1 + q} (-1)^k H^{d - 1,q}(H^k_c X^\circ)$.
Lemma~\ref{l:affine}, Proposition~\ref{p:Gysin} and \eqref{e:Gysin} 
 imply the first equality, and the equation
\[
\sum_q e^{d - 1,q}_G (X^\circ) = e^{d - 1,0}_G (X^\circ) + e^{d - 1,d - 1}_G (X^\circ) = e^{d - 1,0}_G (X^\circ) + 1. 
\]
On the other hand, Theorem~\ref{t:xycharacteristic} implies that the representations $\phi_i$ are effective and $\sum_q e^{d - 1,q}_G (X^\circ) = (-1)^{d - 1} \det(\rho) \cdot \phi_{d} + 1$.  We conclude that $e^{d - 1,0}_G (X^\circ) = (-1)^{d - 1} \det(\rho) \cdot \phi_{d}$, and the result follows from Corollary~\ref{c:reciprocity}. 
\end{proof}


Our next goal is to prove several corollaries which will be useful for proving parts of the equivariant mirror symmetry conjecture in Section~\ref{s:mirror}. 
Recall that $Y$ is the toric variety defined by the normal fan to $P$, and recall that if $G$ acts on a set $S$, then we write $\chi_{\lef S \rig}$ for the corresponding permutation representation. Let $\Phi_k$ denote the lattice points 
in $P$ which lie in the relative interior of a $k$-dimensional face of $P$.

For the remainder of the section, we consider the following setup:

\emph{ Let  $\P = \P(\Sigma)$  be a 
complete 
toric variety with at worst quotient singularities and with a  $G$-action via toric morphisms, admitting a $G$-invariant, proper, birational morphism $f: \P \rightarrow Y$. Let $X'$ denote the closure of $X^\circ$ in $\P$. }

We briefly recall the notation and results from \ref{step2}. That is,  for every cone $\tau'$ in $\Sigma$, let $f(\tau') = Q$, where the normal cone $\tau = \tau_Q$ to $Q$ is 
the smallest cone in the normal fan 
to $P$ containing $\tau'$. Then
\begin{equation}\label{e:coffee}
E_G (X') = \sum_{[\tau'] \in \Sigma /G} \Ind_{G_{\tau'}}^G [E_{G_{\tau'}}(X \cap T_{f(\tau')})E_{G_{\tau'}}( T_{\tau',f})],
\end{equation}
where  $\Sigma/G$ denotes the set of orbits of $G$ acting on the cones in $\Sigma$, $\tau'$ denotes a representative of an orbit, $G_{\tau'}$ denotes the isotropy group of $\tau'$, and 
$T_{\tau',f} = \Spec \C[M_{\tau'}/M_{\tau}]$. 

In the case when $G$ is trivial, the corollary below follows from  
Proposition~5.8 and its proof in \cite{DKAlgorithm} (cf. Corollary~\ref{c:p0}).

\begin{corollary}\label{c:zerosmooth}
With the notation above, 
\[
H^{d - 1,0}(X') = 
\det(\rho) \cdot \chi^*_P,
\]
 and 
\[
H^{p,0}(X') 
= 0 \: \: \textrm{ for } 0 < p < d - 1.
\]
 Moreover,
$e^{d - 2,0}_G (X^\circ) = (-1)^{d - 1} \det(\rho) \cdot \chi_{\lef \Phi_{d - 1} \rig}$ for $d \ge 3$.
\end{corollary}
\begin{proof}
After comparing coefficients of $u^{d - 1}$ on both sides of \eqref{e:coffee}, the first claim follows from 
Corollary~\ref{c:zero}. 
It follows from \eqref{e:Gysin} that $u^{d - 1 - p}v^{d - 1}$ does not appear as a coefficient in the right hand side of \eqref{e:coffee}. The second claim now follows since $E_G (X'; u,v) = (uv)^{d - 1} 
E_G (X';u^{-1}, v^{-1})$ by Example~\ref{e:smooth}. 
Comparing coefficients of $u^{d - 2}$ on both sides of \eqref{e:coffee} yields
\[
0 = e^{d - 2,0}_G (X^\circ) +  \sum_{  \substack{[Q] \in P/G \\ \dim Q = d - 1}  } \Ind_{G_Q}^{G} 
e_{G_Q}^{d - 2,0} (X \cap T_Q),  
\]
where $P/G$ denotes the set of $G$-orbits of faces of $P$. The result now follows from Corollary~\ref{c:zero}, using the fact that if $g$ in $G$ fixes a facet $Q$ of $P$, then $\det \rho(g) = \det \rho_Q(g)$. 
\end{proof}

For any face $Q$ of $P$, let $G_Q$ denote the isotropy subgroup of $Q$. 
As in Section~\ref{s2}, let $M^Q$ be a translate of the intersection of the affine span of $Q$ with $M'$ to the origin, with corresponding representation $\rho_Q: G_Q \rightarrow GL(M^Q)$. 
For each non-negative integer $r$, we define a representation 
\begin{equation}\label{e:theta}
\theta(r) = \theta_{\Sigma}(r) = \sum_{\substack{[Q] \in P / G \\ \dim Q = r}} \Ind_{G_Q}^G [ \det(\rho_Q) \cdot 
\chi^*_Q \cdot  \chi^*_{\tau_Q} ],       
\end{equation}
where $\chi^*_{\tau_Q}$   
denotes the permutation representation induced by the action of $G_Q$ on 
all rays in $\Sigma$ which lie in  the relative interior of the normal cone $\tau_Q$ to $Q$. 

\begin{corollary}\label{c:compute2}
With the notation above, if $S(\Sigma)$ denotes the set of rays of 
$\Sigma$ not lying in the interior of a maximal cone of the normal fan to $P$ and $d \ge 3$, then 
the non-primitive part of the $G$-representation $H^{1,1}(X')$ equals 
\[
\theta(1) + \chi_{\lef S(\Sigma) \rig} - \rho.  
\]
\end{corollary}
\begin{proof}
By Theorem~\ref{t:xycharacteristic} and Corollary~\ref{c:reciprocity}, if $P$ is $1$-dimensional, then 
\[
e^{0,0}_G (X^\circ) = e^{0,0}_G (X) = 1 + \det(\rho) \cdot \chi_P^*.  
\]
Hence, if we compare coefficients of $(uv)^{d - 2}$ on both sides  of \eqref{e:coffee}, we obtain 
the following expression for $e^{d - 2, d - 2}_G(X')$
\[
e^{d - 2, d - 2}_G(X^\circ) +  
 \sum_{\substack{[\tau'] \in \Sigma /G \\ \dim \tau' = 1, \dim f(\tau') > 0  }} \Ind_{G_{\tau'}}^G 1
 +  \sum_{\substack{[\tau'] \in \Sigma /G \\ \dim \tau' = 1, \dim f(\tau') = 1  }} \Ind_{G_{\tau'}}^G 
 \det(\rho_{f(\tau')}) \cdot \chi_{f(\tau')}^*.
\]
By \eqref{e:Gysin}, we obtain
\[
e^{d - 2, d - 2}_G(X') = -\rho + \chi_{\lef S(\Sigma) \rig} + \theta(1),
\]
as desired. 
\end{proof}

\excise{
\begin{remark}\label{r:compute}
With the notation of the above proof, 
if we compare coefficients of $(uv)^{d - 2}$ on both sides  of \eqref{e:coffee}, we obtain 
\[
e^{d - 2, d - 2}_G(X') = e^{d - 2, d - 2}_G(X^\circ) +  \chi_{\lef \Phi_{1} \rig} + 
 \sum_{\substack{\tau' \in \Sigma /G \\ \dim \tau' = 1, \dim f(\tau') < d - 1 }} \Ind_{G_{\tau'}}^G 1.
\]
By \eqref{e:Gysin} and for $d \ge 3$, we conclude that the non-primitive part of the $G$-representation $H^{1,1}(X')$ is the sum of $- \rho$, $\chi_{\lef \Phi_{1} \rig}$,  and the permutation representation induced by the action of $G$ on the rays of 
$\Sigma$ not lying in the interior of a maximal cone of the normal fan to $P$.  
\end{remark}
}

In the case when $G$ is trivial, 
the corollary below  is 
Corollary~5.9 in \cite{DKAlgorithm}. 
Recall that $\Phi_k$ denotes the lattice points 
in $P$ which lie in the relative interior of a $k$-dimensional face of $P$.

\begin{corollary}\label{c:onedown}
For $d \ge 4$,
\[
e^{d - 2, 1}_G (X^\circ) = (-1)^{d - 1} \det(\rho) \cdot [ \phi_{d - 1} -  \chi_{\lef \Phi_{d - 1} \rig}].
\]
\end{corollary}
\begin{proof}
By Theorem~\ref{t:xycharacteristic} and \eqref{e:Gysin}, 
\[
e^{d - 2, 0}_G (X^\circ) + e^{d - 2, 1}_G (X^\circ)  + e^{d - 2, d - 2}_G (X^\circ) = - \rho + (-1)^{d - 1} \det(\rho) \cdot \phi_{d - 1}. 
\]
Since $e^{d - 2, d - 2}_G (X^\circ) = - \rho$ by  \eqref{e:Gysin}, the result follows from Corollary~\ref{c:zerosmooth}. 
\end{proof}

\begin{remark}\label{r:onedown}
When $d = 3$, the above proof shows that Corollary~\ref{c:onedown} holds provided one only considers the contribution  to $e^{1, 1}_G (X^\circ)$ from primitive cohomology. 
\end{remark}

\begin{corollary}\label{c:real}
With the notation above and for
$d \ge 3$, 
the primitive part of the $G$-representation $H^{d- 2,1}(X')$ equals 
\[
\det(\rho) \cdot [ \phi_{d - 1} -  \chi_{\lef \Phi_{d - 1} \rig}] + \theta(d - 2).
\]
\end{corollary}
\begin{proof}
If we compare coefficients of $u^{d - 2}v$ on both sides  of \eqref{e:coffee}, we obtain 
\[
e^{d - 2, 1}_G(X') = e^{d - 2, 1}_G(X^\circ) +  
 \sum_{\substack{[\tau'] \in \Sigma /G \\ \dim \tau' = 1, \dim f(\tau') = d  - 2  }} \Ind_{G_{\tau'}}^G e^{d - 3,0}_{G_{\tau'}}(X \cap T_{f(\tau')}).
\]
By Corollary~\ref{c:zero}, the latter term in the above sum is $(-1)^{d - 1}\theta(d - 2)$. The result now follows from Corollary~\ref{c:onedown} and Remark~\ref{r:onedown}. 
\end{proof}

\section{Applications for simple polytopes}\label{s:cohomology}

In this section, we specialize to the case when $P$ is a simple polytope. 
We continue with the notation of Section~\ref{s2} and Section~\ref{s:ehd}. In the case when $G$ is trivial, these results are due to
Danilov and Khovanski{\u\i}  \cite{DKAlgorithm}.

We assume throughout this section that $P$ is \define{simple}. That is, we assume that every 
 vertex of $P$ is adjacent to precisely $d$ facets. 
Equivalently, $P$ is simple if and only if the  toric variety $Y$ corresponding to the normal fan of $P$ has at worst quotient singularities. Let $X^\circ \subseteq T$ be a $G$-invariant, non-degenerate hypersurface with Newton polytope $P$, and let $X = X_P$ be the closure of $X^\circ$ in $Y$. If $P$ is simple, then  $X$ itself has at worst quotient singularities, and hence $H^k X = \bigoplus_{p + q = k} H^{p,q}(X)$ admits a pure Hodge structure of weight $k$. 
The Lefschetz hyperplane theorem implies that the restriction map
$H^k(Y) \rightarrow H^k(X)$
is a $G$-equivariant isomorphism for $k < d - 1$, and an injection for $k = d - 1$. Since 
 Poincar\'e duality induces isomorphisms of $G$-representations $H^{p,q}(X) \cong H^{d - 1 - p, d - 1 - q}(X)$
  (Example~\ref{e:smooth}), Example~\ref{e:toric} implies that in order to understand the action of $G$ on $H^* X$, it remains to compute the $G$-representations
  \[
H^{d - 1}_{\prim} X = \bigoplus_p H^{p, d - 1 - p}_{\prim} (X) := \coker[
H^{d - 1} Y \rightarrow H^{d - 1} X].
\]
 In fact, since 
  we have isomorphisms of $G$-representations $H^{p,q}(X) \cong H^{q,p}(X)$ (Remark~\ref{r:symmetry}), it is enough to compute  $H^{p, d - 1 - p}_{\prim} (X)$ for $p \ge \frac{d - 1}{2}$. 

For any face $Q$ of $P$, let $G_Q$ denote the isotropy subgroup of $Q$. 
As in Section~\ref{s2}, let $M^Q$ be a translate of the intersection of the affine span of $Q$ with $M'$ to the origin, with corresponding representation $\rho_Q: G_Q \rightarrow GL(M^Q)$.  In the case when $G$ is trivial, the theorem below is proved in Section~5.5 of \cite{DKAlgorithm}. 

\begin{theorem}\label{t:simple}
If $P$ is simple and $p \ge \frac{d - 1}{2}$, then
\[
H^{p, d - 1 - p}_{\prim}(X) = \sum_{[Q] \in P/G} (-1)^{d - \dim Q} 
\Ind_{G_Q}^{G} [ \det(\rho_Q) \cdot \phi_{Q, p + 1} ], 
\]
where $P/G$ denotes the set of $G$-orbits of faces of $P$. 
\end{theorem}
\begin{proof}
With the notation of Section~\ref{toric}, $X$ admits a $G$-invariant stratification $X = \coprod_{Q \subseteq P} X \cap T_Q$. Hence Proposition~\ref{p:induced} implies that 
\[
E_G(X) = \sum_{[Q] \in P/G} \Ind_{G_Q}^{G} E_{G_Q}(X \cap T_Q). 
\]
By the discussion above, $e^{p,q}_G(X) = (-1)^{p + q} H^{p,q}(X) =  0$ unless $p = q$ or $p + q = d - 1$, 
and we compute, using  \eqref{e:Gysin}, 
\begin{equation}\label{e:eqn1}
\sum_q e^{p,q}_G(X) = (-1)^{d - 1} H_{\prim}^{p,d - 1 - p}(X) +  \sum_{[Q] \in P/G} \Ind_{G_Q}^{G} e^{p + 1, p + 1}_{G_Q}(T_Q). 
\end{equation}
On the other hand, Theorem~\ref{t:xycharacteristic} and \eqref{e:Gysin} imply that 
\[
\sum_q e^{p,q}_{G_Q} (X \cap T_Q) = e^{p + 1, p + 1}_{G_Q}(T_Q) + (-1)^{\dim Q - 1} \det(\rho_Q) \cdot \phi_{Q,p + 1},
\]
and we deduce that 
\begin{equation}\label{e:eqn2}
\sum_q e^{p,q}_G(X) = \sum_{[Q] \in P/G} \Ind_{G_Q}^{G} [e^{p + 1, p + 1}_{G_Q}(T_Q) + (-1)^{\dim Q - 1} \det(\rho_Q) \cdot \phi_{Q,p + 1}]. 
\end{equation}
Comparing \eqref{e:eqn1} and \eqref{e:eqn2} now yields the desired result. 

\end{proof}

The first statement in the corollary below also follows from Corollary~\ref{c:zerosmooth}.

\begin{corollary}\label{c:genus}
With the notation above, if $P$ is simple, then 
\[
H^{d - 1,0}(X) = \det(\rho) \cdot \chi_P^*. 
\]
In particular, 
\[
\sum_{m \ge 0} H^{d - 1,0}(X_{mP}) t^m = \det(\rho) \cdot  \phi[t] \cdot \sum_{m \ge d + 1} \Sym^{m - d - 1} (\rho + 1) \;  t^{m}. 
\]
\end{corollary}
\begin{proof}
The first statement is an immediate consequence of Theorem~\ref{t:simple}, using the fact that $\phi_{Q,i} = 0$ for $i > \dim Q$. 
 The second statement follows from 
Lemma~\ref{l:exterior} and Corollary~\ref{c:reciprocity}. 
\end{proof}

\begin{remark}\label{r:genus}
Corollary~\ref{c:genus} and Lemma~\ref{l:permutation} together imply that $\dim H^{d - 1,0}(X)^G = \dim H^{d - 1,0}(X/G)$ equals the number of $G$-orbits of $\Int(P) \cap M$  whose isotropy subgroup is contained in    $\det(\rho)^{-1}(1)$. 
\end{remark}

\begin{remark}\label{r:reciprocity}
Recall that $P$ corresponds to a projective toric variety $Y$ and ample line bundle $L$, and that we have equality of $G$-representations $H^0 (Y, L^{\otimes m}) = \chi_{mP}$. If we set 
$a(m) = \dim H^{d - 1,0}(X_{mP}/G)$ and $b(m) = \dim H^0 (Y, L^{\otimes m})^G$, then Corollary~5.7 in \cite{YoEquivariant} implies that $a(m)$ and $b(m)$ are  quasi-polynomials in $m$ of degree $d$, with leading coefficient $\frac{\vol P}{|G|}$ and period dividing the  exponent of $G$. Moreover, the quasi-polynomials satisfy the reciprocity relation $a(m) = (-1)^d b(-m)$. 
\end{remark} 

\begin{example}[Fermat hypersurfaces]\label{e:Fermat}
Let $G = \Sym_{d + 1}$ act on $\Z^{d + 1}$ by permuting co-ordinates, and let $P$ be the standard $d$-dimensional simplex with vertices $\{ e_0, \ldots, e_d \}$. Then  $M \cong \Z^{d + 1}/\Z(1, \ldots, 1)$, 
$\rho: G \rightarrow GL(M)$ is the reflection representation, and one verifies that $\phi[t] = 1$
(cf. \cite[Proposition~6.1]{YoEquivariant}). 

In this case, the Fermat hypersurface $X_m = \{ x_0^m + \cdots + x_d^m = 0 \} \subseteq \P^d$ of degree $m$ is a non-degenerate, $G$-invariant  hypersurface corresponding to the polytope $mP$.  We deduce from Corollary~\ref{c:genus} that
\[
H^{d - 1,0}(X_{m}) = \sgn \cdot  \Sym^{m - d - 1} (V),
\]
where $\sgn$ is the $1$-dimensional sign representation, and $\Sym_{d + 1}$ acts on $V = \C^{d + 1}$
by permuting co-ordinates. Moreover, $\dim H^{d - 1,0}(X_{m}/G)$ equals the number of partitions of $m$ with $d + 1$ distinct parts, and $\dim H^0 (\P^d, \mathcal{O}(m))^G$ equals the number of partitions of $m$ with at most $d + 1$ parts. 
 In this case, 
the reciprocity result in Remark~\ref{r:reciprocity} is a classical result on partitions \cite[Theorem~4.5.7]{StaEnumerative}.
\end{example}

\begin{example}[Fermat curves]\label{e:curves}
Letting $d = 2$ in the example above, we obtain the action of $\Sym_3$ on the Fermat curve
 $C_m = \{ x^m + y^m + z^m = 0 \}$ of degree $m$. If $\zeta$ denotes the $2$-dimensional reflection representation, then the irreducible representations of $\Sym_3$ are $\{ 1, \sgn, \zeta \}$. 
 Using the above results, one explicitly computes that if $\nu_r(m)$ denotes the function with value $1$ if $r | m$, and value $0$ otherwise, then
\[
H^{1,0}(C_{m}) = \frac{(m - 1)(m - 5)}{12} + \frac{\nu_2(m)}{4} + \frac{\nu_3(m)}{3} + 
\]
\[
\left[ \frac{m^2 - 1}{12} - \frac{\nu_2(m)}{4} + \frac{\nu_3(m)}{3} \right] \sgn + 
\left[ \frac{(m - 1)(m - 2)}{6} - \frac{\nu_3(m)}{3}\right] \zeta. 
\] 
In particular, $C_m/G$ is a smooth, rational curve if and only if $m \le 5$ (cf. \cite[Example~1.3]{DLPurity}). 
\end{example}

Our next goal is to determine a formula for $e^{p,q}_G(X^\circ)$ when $P$ is simple.
If $B$ is a finite poset, then
the M\"obius function $\mu_{B}: B \times B \rightarrow \mathbb{Z}$ is defined recursively by,
\[
\mu_{B}(x,y) =  \left\{ \begin{array}{ll}
1 & \textrm{  if  } x = y \\
0 & \textrm{  if  } x > y \\
- \sum_{x < z \leq y} \mu_{B}(z,y) = - \sum_{x \leq z < y} \mu_{B}(x,z) & \textrm{  if  } x < y, 
\end{array} \right.
\]
and satisfies the property (known as `M\"obius inversion') that for any function $h: B \rightarrow A$ to an abelian group $A$, 
\begin{equation}\label{inversion}
h(z) = \sum_{y \le z} \mu_{B}(y,z) g(y), \textrm{ where } g(y) = \sum_{x \leq y} h(x).
\end{equation}

For any face $Q$ of $P$, recall that we have representations $\rho_Q: G_Q \rightarrow GL(M^Q)$, where $G_Q$ denote the isotropy subgroup of $Q$. 

\begin{lemma}\label{l:poset}
Fix an element $g$ in $G$, and let $B$ be the poset of (non-empty) $g$-fixed faces of $P$. Then 
 $\mu_B(Q,P)= (-1)^{d - \dim Q} \det \rho(g) \det \rho_Q(g)$. 
\end{lemma}
\begin{proof}
Let $N_Q$ be the sublattice of $N = \Hom(M,\Z)$ spanned by the normal cone to $Q$. We have an isomorphism of lattices $N_Q \cong M/M^Q$ such that if
$g$ acts on $M/M^Q$ via an integer matrix $A$, then $g$ acts on $N_Q$ via the inverse transpose of $A$. 
If $\{ \lambda_i \}$ denote the eigenvalues of $A$, then the eigenvalues of $A^{-1}$ are the conjugates 
$\{ \bar{\lambda_i} \}$. Since $A$ is integer valued, we conclude that $A$ and the inverse transpose of $A$ have the same eigenvalues and hence we have an isomorphism of $G_Q$-representations $(M/M^Q)_\C \cong (N_Q)_\C$. In particular, $\rho = \rho_Q + (N_Q)_\C$ in $R(G_Q)$, and 
$\det \rho(g) \det \rho_Q(g) = \det (N_Q)_\C(g)$.

On the other hand, since $P$ is simple, if $Q$ has codimension $n$ in $P$, then  $Q$ lies in precisely 
$n$ facets $\{ F_1, \ldots, F_n \}$ of $P$, and $(N_Q)_\C$ is the permutation representation induced by the action of $G$ on these facets. 
Let $\{ V_1, \ldots, V_s \}$ denote the set of $g$-orbits of $\{ F_1, \ldots, F_n \}$. 
For any (possibly empty) subset $I \subseteq \{ 1, \ldots , s \}$, let $Q_I$  be the
intersection of the facets $\{ F_j \in V_i \mid i \in I \}$. 
Then the faces $\{ Q_I \mid I \subseteq \{ 1, \ldots , s \} \}$ are precisely the faces of $P$ which contain $Q$ and are fixed by $g$. 
 Since $d - \dim Q_I = \sum_{i \in I} |V_i|$, and 
$\det (N_Q)_\C(g) = (-1)^{\sum_{i \in I} (|V_i| - 1)}$, we conclude that 
$(-1)^{d - \dim Q} \det \rho(g) \det \rho_{Q_I} (g) = (-1)^{|I|}$. The result now follows by induction on $| I |$, and the fact that $\sum_{I \subseteq \{ 1, \ldots , s \}}  (-1)^{|I|} = 0$. 
\end{proof}

We are now ready to compute $e^{p,q}_G (X^\circ)$. Since $e^{p,q}_G (X^\circ) = e^{q,p}_G (X^\circ)$
(Remark~\ref{r:symmetry}), and $\sum_q e^{p,q}_G(X^\circ)$ is computed by Theorem~\ref{t:xycharacteristic}, we may and will assume that $p > q$. Recall that $G_Q$ denotes the isotropy group of a face $Q$ of $P$.  In the case when $G$ is trivial, the theorem below is Theorem~5.6 in \cite{DKAlgorithm}.

\begin{theorem}\label{t:simpleint}
If $P$ is simple and $p > q$, then $(-1)^{d + p + q} e^{p,q}_G (X^\circ)$ equals 
\begin{displaymath}
\det (\rho) \cdot \sum_{\substack{[Q] \in P/G \\ \dim Q = p + q + 1}} \Ind_{G_Q}^{G} 
\left[  \det (\rho_Q) \cdot  \sum_{[Q'] \in Q/G_Q} (-1)^{\dim Q'} \Ind_{G_{Q'}}^{G_Q} [\det (\rho_{Q'})  
\cdot \phi_{Q', p + 1}] \right], 
\end{displaymath}
where $P/G$ denotes the set of $G$-orbits of faces of $P$, and
$Q/G_Q$ denotes the set of $G_Q$-orbits of faces of $Q$. 
\end{theorem}
\begin{proof}
If we fix $g$ in $G$, then by Proposition~\ref{p:induced},
\[
e^{p,q}_G (X)(g) = \sum_{g \cdot Q = Q} e^{p,q}_{G_Q}(X \cap T_Q)(g), 
\]
where $T_Q$ denotes the torus orbit corresponding to $Q$. Let $X_Q$ denote the closure of 
$X \cap T_Q$ in $X$. 
Applying M\"obius inversion
to the poset of $g$-fixed faces of $P$ using Lemma~\ref{l:poset} yields 
\[
e^{p,q}_G (X^\circ)(g) = \sum_{g \cdot Q = Q} (-1)^{d - \dim Q} \det \rho(g) \det \rho_Q(g)
e^{p,q}_{G_Q}(X_Q)(g).
\]
By Proposition~\ref{p:Gysin} and Example~\ref{e:smooth}, the assumption $p > q$ implies that $e^{p,q}_{G_Q}(X_Q) = 0$ unless $p + q = \dim Q - 1$, in which case Theorem~\ref{t:simple} implies that
\[
e^{p,q}_{G_Q}(X_Q)(g) = \sum_{\substack{Q' \subseteq Q \\ g \cdot Q' = Q'}} (-1)^{\dim Q' - 1}
\det \rho_{Q'}(g) \phi_{Q', p + 1}(g). 
\]  
Putting this together yields the theorem. 
\end{proof}

We immediately obtain the following corollary (cf. Corollary~\ref{c:zero} and Corollary~\ref{c:zerosmooth}).  Recall that if $G$ acts on a set $S$, then we write $\chi_{\lef S \rig}$ for the corresponding permutation representation. Let $\Phi_k$ denote the lattice points 
in $P$ which lie in the relative interior of a $k$-dimensional face of $P$. 
\begin{corollary}\label{c:p0}
With the notation above, if $P$ is simple, then for any $p > 0$,
\[
e^{p,0}_G (X^{\circ}) = (-1)^{d - 1} \det(\rho) \cdot \chi_{\lef \Phi_{p + 1} \rig},
\]
and 
\[
e^{0,0}_G (X^{\circ}) = (-1)^{d - 1} \det(\rho) \cdot [\chi_{\lef \Phi_{1} \rig} + \chi_{\lef \Phi_{0} \rig} - 1].
\]
\end{corollary}
\begin{proof}
If we fix $g$ in $G$ and $p > 0$, then Theorem~\ref{t:simpleint} implies that $e^{p,0}_G (X^{\circ})(g)$ equals
\[
\sum_{\substack{g \cdot Q = Q \\ \dim Q = p + 1}} (-1)^{d - p} \det \rho(g) \det \rho_Q(g)
 \sum_{\substack{Q' \subseteq Q \\ g \cdot Q' = Q'}} (-1)^{\dim Q'}
\det \rho_{Q'}(g) \phi_{Q', p + 1}(g). 
\]
Corollary~\ref{c:reciprocity} and Corollary~\ref{c:vanishing} imply that $\phi_{Q', p + 1} = 0$ if $\dim Q' < p + 1$, and $\phi_{Q, p + 1}$ equals the number of $g$-fixed lattice points in the relative interior of $Q$. 
This proves the first statement. For the second statement, Theorem~\ref{t:xycharacteristic} implies that
\[
\sum_q e^{0,q}_G(X^\circ) = (-1)^{d - 1} \bigwedge^{d - 1} \rho + (-1)^{d - 1}\det(\rho) \cdot \phi_{P,1}. 
\]
By Lemma~\ref{l:dual} and the first statement, we obtain
\[
e^{0,0}_G(X^\circ) = (-1)^{d - 1}\det(\rho) \cdot [ \phi_{P,1} + \rho - \sum_{k > 1} \chi_{\lef \Phi_{k} \rig}]. 
\]
By definition, $\phi_{P,1} = \sum_{k} \chi_{\lef \Phi_{k} \rig} - \rho - 1$, and the result follows. 
\end{proof}

\begin{remark}
In the case when $G$ is trivial, the above corollary is Proposition~5.8 in \cite{DKAlgorithm}, and is proved without the assumption that $P$ is simple. It would be interesting to extend the above corollary to the general case (cf. Corollary~\ref{c:zero} and Corollary~\ref{c:zerosmooth}). 
\end{remark}

\section{Applications for simplices}\label{s:simplex}

In this section, we further specialize to the case when $P$ is a simplex, and present an explicit example of the representation of the symmetric group acting on the cohomology of a Fermat hypersurface. 

We continue with the notation of Section~\ref{s2} and Section~\ref{s:cohomology}.
That is, let $X^\circ \subseteq T$ be a $G$-invariant, non-degenerate hypersurface with Newton polytope $P$, and let $X = X_P$ be the closure of $X^\circ$ in the toric variety $Y$ determined by the normal fan to $P$.
Throughout this section, we assume that  $P$ is a \define{simplex} i.e. $P$ has precisely $d + 1$ vertices $\{ v_0, \ldots, v_d \}$. 
For each face $Q$ of $P$, let $\Pi(Q)$ denote the set of interior lattice points of the 
parallelogram spanned by the vertices $\{ (v_i, 1 ) \mid v_i \in Q \}$ of $Q \times 1$ in $M \oplus \Z$. 
That is,
\[
\Pi(Q) = \{ w \in M \oplus \Z \mid w = \sum_{v_i \in Q} \alpha_i (v_i, 1) \textrm{ for some } 0 < \alpha_i < 1 \}.
\]
We set $\Pi(Q) = \{ 0 \}$ when $Q$ is the empty face. 
Let $u: M \oplus \Z \rightarrow \Z$ denote projection onto the second co-ordinate, and let
$\Pi(Q)_k = \{ w \in \Pi(Q) \mid u(w) = k \}$.
Recall that if $G$ acts on a set $S$, then we write $\chi_{\lef S \rig}$ for the corresponding permutation representation.
The result below is due to Batyrev and Nill in the case when $G$ is trivial \cite[Proposition~4.6]{BNCombinatorial}.

\excise{
\[
\Pi(Q) = \{ w \in M' \mid w = \sum_{v_i \in Q} \alpha_i (v_i, 1) \textrm{ for some } 0 < \alpha_i < 1 \},
\]
and let 
$\Pi(Q)_k = \{ v \in \Pi(Q) \mid u(v) = k \}$.
Let $\Pi$ denote the set of interior lattice points of the 
parallelogram spanned by the vertices $\{ (v_i, 1 ) \}$ of $P \times 1$ in $M' = M \oplus \Z$. That is, 
\[
\Pi = \{ w \in M' \mid w = \sum_{i = 0}^d \alpha_i (v_i, 1) \textrm{ for some } 0 < \alpha_i < 1 \}. 
\]
Let $u: M' = M \oplus \Z \rightarrow \Z$ denote projection onto the second co-ordinate, and let
$\Pi_k = \{ v \in \Pi \mid u(v) = k \}$.
Recall that if $G$ acts on a set $S$, then we write $\chi_{\lef S \rig}$ for the corresponding permutation representation.
The result below is due to Batyrev and Nill in the case when $G$ is trivial \cite[Proposition~4.6]{BNCombinatorial}. 
}

\begin{corollary}\label{c:simplex}
With the notation above, if $P$ is a simplex, then $H^{p, d - 1 - p}_{\prim} (X) = \det(\rho) \cdot \chi_{\lef \Pi(P)_{p + 1} \rig}$.
In particular, 
$H^{d - 1}_{\prim} (X) = \det(\rho) \cdot \chi_{\lef \Pi(P) \rig}$.
\end{corollary}
\begin{proof}
Since $G$ permutes the vertices of $P$,
$\Pi(P)$ admits a $G$-equivariant involution 
\[
\iota: \Pi(P) \rightarrow \Pi(P), \: \; \: w = \sum_{i = 0}^d \alpha_i (v_i, 1) \mapsto \iota(w) = \sum_{i = 0}^d (1 - \alpha_i) (v_i, 1),
\]
satisfying $u(w) + u(\iota(w)) = d + 1$. Since we have an equality of $G$-representations
$H^{p, d - 1 - p}_{\prim} (X) = H^{d - 1 - p, p}_{\prim} (X)$ (Remark~\ref{r:symmetry}), it follows that we may reduce the proof to the case when $p \ge \frac{d - 1}{2}$. 
In this case, for a fixed $g$ in $G$, Theorem~\ref{t:simple} implies that
\begin{equation*}
H^{p, d - 1 - p}_{\prim} (X)(g) = \sum_{g \cdot Q = Q} (-1)^{d - \dim Q} \det \rho_Q(g)\phi_{Q, p + 1}(g).
\end{equation*}

 Then Proposition~\ref{p:simplex} implies that 
 \[
 \phi_{Q,p + 1}(g) = \sum_{\substack{Q' \subseteq Q \\  g \cdot Q' = Q' }} \chi_{\lef \Pi(Q')_{p + 1} \rig}(g), 
 \]
 and hence
 \begin{align*}
 H^{p, d - 1 - p}_{\prim} (X)(g) &= \sum_{g \cdot Q = Q} (-1)^{d - \dim Q} \det \rho_Q(g) \sum_{\substack{Q' \subseteq Q \\  g \cdot Q' = Q' }} \chi_{\lef \Pi(Q')_{p + 1} \rig}(g) \\
 &= \det \rho(g) \sum_{g \cdot Q' = Q'}   \chi_{\lef \Pi(Q')_{p + 1} \rig}(g)    
 \sum_{\substack{Q' \subseteq Q \\  g \cdot Q = Q}}   (-1)^{d - \dim Q} \det \rho(g) \det \rho_Q(g).
 \end{align*}
 By Lemma~\ref{l:poset}, the final summand in the above expression is $1$ if $Q' = P$ and $0$ otherwise, and the first statement follows. The second statement is immediate.

\end{proof}

We define $\Pi(r) = 
 \coprod_{\dim Q = r} \Pi(Q)$ and $\Pi(r)_k 
 = \coprod_{\dim Q = r} \Pi(Q)_k$. 

\begin{corollary}\label{c:simplex2}
With the notation above, if $P$ is a simplex and $p > q$, then
\[
e_G^{p,q}(X^\circ) = (-1)^{d - 1} \det(\rho) \cdot \chi_{\lef 
 \Pi(p + q + 1)_{p + 1} \rig}.
\]
\end{corollary}
\begin{proof}
Recall from the proof of Theorem~\ref{t:simpleint} that, for a fixed $g$ in $G$,
\[
e^{p,q}_G (X^\circ)(g) = \sum_{\substack{g \cdot Q = Q \\ \dim Q = p + q + 1}} (-1)^{d - \dim Q} \det \rho(g) \det \rho_Q(g)
e^{p,q}_{G_Q}(X_Q)(g).
\]
By Corollary~\ref{c:simplex}, since $p + q= \dim Q - 1$ and $p > q$, 
\[
e^{p,q}_{G_Q}(X_Q)(g) = (-1)^{\dim Q - 1} H^{p,q}_{\prim} (X_Q) = 
(-1)^{\dim Q - 1} \det \rho_Q(g)\chi_{\lef \Pi(Q)_{p + 1} \rig}(g),
\] 
and the result follows. 
\end{proof}

\begin{remark}\label{r:quotient}
Assume that $P$ is a simplex, and let $H^{d - 1}_{\prim}(X/G)$ 
denote the subspace of $H^{d - 1}(X/G)$  
 corresponding to  
$H^{d - 1}_{\prim} (X)^G$ under the 
isomorphism $H^*(X/G) \cong H^* (X)^G$, with its corresponding pure Hodge structure. 
 Then Corollary~\ref{c:simplex} and Lemma~\ref{l:permutation} imply that  $\dim H^{p, d - 1 - p}_{\prim} (X/G)$ equals the number of $G$-orbits of $\Pi(P)_{p + 1}$ whose isotropy subgroup is contained in
 $\det(\rho)^{-1}(1)$. 

Deducing the dimensions of the pieces  of the Hodge structure on the cohomology of $X/G$ then reduces to determining the numbers $\dim H^{2i}(Y)^G$, where $Y$ is the toric variety corresponding to $P$. The latter can be computed using the fact that 
$e^{p,p}_G(Y) = H^{p,p}(Y)$, and using the formula for $E_G(Y)$ 
from Example~\ref{e:toric}. 



\end{remark}

\begin{example}[Fermat hypersurfaces]\label{e:Fermat2}
We continue with the notation of Example~\ref{e:Fermat}. That is, let $G = \Sym_{d + 1}$ act on $\Z^{d + 1}$ by permuting co-ordinates, and let $P$ be the standard $d$-dimensional simplex with vertices $\{ e_0, \ldots, e_d \}$. Then  the Fermat hypersurface $X_m = \{ x_0^m + \cdots + x_d^m = 0 \} \subseteq \P^d$ of degree $m$ is a non-degenerate, $G$-invariant  hypersurface corresponding to the polytope $mP$.  Corollary~\ref{c:simplex} implies that $\sgn \cdot H^{p, d - 1 - p}_{\prim}(X_m)$ is isomorphic to the permutation representation of $\Sym_{d + 1}$ on the set 
\[
\{ (a_0, \ldots, a_d) \in \Z^{d + 1} \mid 0 < a_i < m, \sum_{i = 0}^d a_i = (p + 1)m \}. 
\]
In particular, $\sgn \cdot H^{d - 1}_{\prim}(X_m)$  is isomorphic to the permutation representation of $\Sym_{d + 1}$ on the set 
\begin{equation*}
\{ (a_0, \ldots, a_d) \in (\Z/m\Z)^{d + 1} \mid a_i \ne 0, \sum_{i = 0}^d a_i = 0 \}. 
\end{equation*}
The ring isomorphism $H^*(X_m/G) \cong H^* (X_m)^G$ induces an isomorphism 
\[
H^*(X_m/G) \cong H^*(\P^{d - 1}) \oplus H^{d - 1}_{\prim} (X_m)^G. 
\]
By Remark~\ref{r:quotient}, $\dim H^{d - 1}_{\prim} (X_m)^G$ is equal to the number of partitions 
of multiples of $m$ into $(d + 1)$-distinct parts of size strictly less than $m$. In particular, 
$H^*(X_m/G) \cong H^*(\P^{d - 1})$ for $m < \binom{d + 2}{2}$ (cf. Example~\ref{e:curves}). 

We also obtain a formula for the character of $H^{p, d - 1 - p}_{\prim}(X_m)$. More specifically, if 
$g$ in $\Sym_{d + 1}$ has cycle type $( \lambda_1, \ldots, \lambda_r )$, then, by Lemma~\ref{l:permutation},  
$\tr(g; H^{p, d - 1 - p}_{\prim}(X_m))$ is equal to 
\[
(-1)^{d + 1 - r} 
\# \{ (a_1, \ldots, a_r) \in \Z^{r} \mid 0 < a_i < m, \sum_{i = 0}^d \lambda_i a_i = (p + 1)m \}.
\]
Similarly, we have $\tr(g; H^{d - 1}_{\prim}(X_m))$ equal to 
\[
  (-1)^{d + 1 - r} 
\# \{ (a_1, \ldots, a_r) \in (\Z/m\Z)^{r} \mid a_i \ne 0, \sum_{i = 0}^d \lambda_i a_i = 0 \}.
\]
Alternative formulas for these characters are given by Ch\^enevert \cite[Theorem~2.2, Corollary~2.5]{CheRepresentations}. 
\end{example}

\section{Equivariant mirror symmetry}\label{s:mirror}

In this section, we  conjecture an equivariant version of mirror symmetry for Calabi-Yau hypersurfaces in toric varieties, and prove it in several cases. 
We continue with the notation of Section~\ref{s2} and Section~\ref{s:cohomology}.

Recall that a $d$-dimensional lattice polytope $P$ in $M$ is \emph{reflexive} if the origin is the unique interior lattice point of $P$ and every non-zero lattice point in $M$ lies in the boundary of $mP$ for some positive integer $m$. Equivalently, if $N = \Hom(M, \Z)$, then $P$ is reflexive if and only if its polar polytope
\[
P^* = \{ u \in N_\R \mid \lef u , v \rig \ge -1, \: \forall \,  v \in P \}
\]
is a lattice polytope. Let $Y$ (respectively $Y^*$) denote the projective toric variety corresponding to the 
normal fan of $P$ (respectively $P^*$). Observe that the normal fan of $P$ is equal to the fan over the faces of $P^*$, and vice versa. If $X$ and $X^*$ denote non-degenerate hypersurfaces in $Y$ and $Y^*$ respectively, then $X$ and $X^*$ are \emph{Calabi-Yau varieties} (see, for example, \cite{BDStrong}). 
In \cite{BDStrong}, Batyrev and Dais associated \define{stringy invariants} $E_{\st}(X; u,v)$ and $E_{\st}(X^*; u,v)$ to $X$ and $X^*$, 
such that if $X$ admits a crepant resolution $\tilde{X} \rightarrow X$, then $E_{\st}(X) = E(\tilde{X})$. 
More precisely, if $\tilde{X} \rightarrow X$ is a resolution of singularities, then 
$E_{\st}(X)$ is   the 
motivic integral associated to the relative canonical divisor on
$\tilde{X}$ \cite{BatNon}.  
Batyrev and Borisov proved the following version of mirror symmetry in \cite{BBMirror},
\begin{equation}\label{e:bb}
E_{\st}(X; u,v) = (-u)^{d - 1} E_{\st}(X^*; u^{-1},v). 
\end{equation}
In particular, if there exist crepant resolutions $\tilde{X} \rightarrow X$ and $\tilde{X}^* \rightarrow X^*$,
then 
\[
\dim H^{p,q}(\tilde{X}) = \dim H^{d - 1 - p,q}(\tilde{X}^*) \: \textrm{  for  } \: 0 \le p,q \le d - 1. 
\]

One may formally extend these definitions to the equivariant setting. More precisely, 
one may define motivic integration for complex varieties with a $G$-action (cf. Section~\ref{s:HodgeDeligne}), and 
then define $E_{\st, G}(X; u, v) \in R(G)[u,v][[u^{-1},v^{-1}]]$ to be the 
(equivariant) motivic integral associated to the relative canonical divisor of an equivariant resolution of singularities
$\tilde{X}$ (see \cite{AWEquivariant}). 
Moreover, if $\tilde{X} \rightarrow X$ is an equivariant, crepant resolution, then  
$E_{\st, G}(X; u, v) = E_{G}(\tilde{X}; u, v)$. 


\begin{conjecture}\label{c:mirror}
Suppose that $G$ acts linearly on a lattice $M$ of rank $d$ via a homomorphism $\rho: G \rightarrow GL(M)$. 
If $P$ and $P^*$ are polar, $G$-invariant, reflexive polytopes, and $X$ and $X^*$ are corresponding $G$-invariant, non-degenerate hypersurfaces, then the \emph{equivariant stringy invariants}
$E_{\st,G}(X; u,v)$ and $E_{\st, G}(X^*; u,v)$ 
 are rational functions satisfying 
\[
E_{\st, G}(X; u,v) = (-u)^{d - 1} \det(\rho) \cdot E_{\st, G}(X^*; u^{-1},v). 
\]
\end{conjecture}


\begin{remark}\label{r:surprise}
Suppose that there exist $G$-equivariant, crepant resolutions 
$\tilde{X} \rightarrow X$ and $\tilde{X}^* \rightarrow X^*$. 
The conjecture implies that 
if $H = \det(\rho)^{-1}(1)$, then the  (possibly singular) Calabi-Yau varieties  
$\tilde{X}/H$ and $\tilde{X}^*/H$ have mirror Hodge diamonds.
Explicitly, if $V^{\det(\rho)}$ denotes the $\det(\rho)$-isotypic component of a $G$-representation $V$, then 
\[
\dim H^{p,q}(\tilde{X}/H) = \dim ( H^{p,q}(\tilde{X})^G + H^{p,q}(\tilde{X})^{\det(\rho)})  = \dim H^{d - 1 - p,q}(\tilde{X}^*/H). 
\]
It would be interesting to know whether $\tilde{X}/H$ and $\tilde{X}^*/H$ are mirror in the usual sense i.e.
whether their associated stringy invariants satisfy \eqref{e:bb}. 
\end{remark}

\begin{remark}
Unlike in the case when $G$ is trivial, there may not exist a $G$-equivariant, crepant, toric morphism $\tilde{Y} \rightarrow Y$ such that $\tilde{Y}$ has orbifold singularities. 
Hence one can not define $E_{\st, G}(X; u,v)$ in terms of the action of $G$ on the orbifold cohomology of an equivariant, partial, crepant resolution \cite{BDStrong}. 
\end{remark}

\begin{remark}
More generally, Batyrev and Borisov proved their mirror symmetry result for Calabi-Yau complete intersections, and 
one could ask for an equivariant version in this case.  In fact, many of our results can be extended to the complete intersection case (see \cite[Section~6]{DKAlgorithm} in the case when $G$ is trivial), although we do not pursue this issue here. 
\end{remark}

A polytope $P$ is \define{smooth} if the toric variety determined by its normal fan is smooth. 
We first prove the conjecture when the polar reflexive polytopes $P$ and $P^*$ are smooth. 

\begin{corollary}\label{c:smoothmirror}
If $P$ and $P^*$ are polar, $G$-invariant, smooth, reflexive polytopes of dimension $d$, and $X$ and $X^*$ are corresponding $G$-invariant, non-degenerate hypersurfaces, then 
\[
H^{p,q}(X) = \det(\rho) \cdot H^{d - 1 - p,q}(X^*) \in R(G) \: \textrm{  for  } \: 0 \le p,q \le d - 1. 
\]
\end{corollary}
\begin{proof}
We first compute the $G$-representation $H^* X = \bigoplus_{p,q} H^{p,q}(X)$. If $Y$ denotes the toric variety corresponding to the normal fan of $P$, then the Lefschetz hyperplane theorem and Poincar\'e duality imply that the non-primitive cohomology of $X$ satisfies $H^{p,p}(X) = H^{p,p}(Y)$ for $p \le \frac{d - 1}{2}$, 
and $H^{p,p}(X) = H^{p + 1,p + 1}(Y)$ for $p \ge \frac{d - 1}{2}$. Corollary~\ref{c:reflexive}
and Proposition~\ref{p:reflexive} then imply that  the non-primitive cohomology of $X$ is given by 
\[
H^{p,p}(X) = \phi_{P^*,p} \textrm{ for } p \le \frac{d - 1}{2}, 
\]
\[
H^{p,p}(X) = \phi_{P^*,p + 1} \textrm{ for } p \ge \frac{d - 1}{2}. 
\] 
On the other hand, every proper face $Q$ of $P$ is isomorphic to a standard simplex, and 
hence Proposition~\ref{p:simplex} implies that $\phi_Q[t] = 1$. Then Theorem~\ref{t:simple}, together with Corollary~\ref{c:reflexive}, implies that 
\[
H^{d - 1 - p, p}_{\prim}(X) = \det(\rho) \cdot  \phi_{P,p}  \textrm{ for } p \le \frac{d - 1}{2}, 
\]
\[
H^{d - 1 - p,p}_{\prim}(X) = \det(\rho) \cdot  \phi_{P,p + 1}  \textrm{ for } p \ge \frac{d - 1}{2}.
\]
The result now follows by symmetry. 
\end{proof}

For the remainder of the section, we assume that 
 both $X$ and $X^*$ admit toric, crepant $G$-equivariant resolutions. 
That is, we assume that there exist $G$-equivariant  lattice polyhedral decompositions of the boundaries of $P$ and  $P^*$ which restrict to smooth, lattice triangulations on faces of $P$ and  $P^*$ of codimension at least $2$. Equivalently, we assume there exists $G$-equivariant, proper, crepant toric morphisms
$\tilde{Y} \rightarrow Y$ and $\tilde{Y}^* \rightarrow Y^*$, such that $\tilde{Y}$ and $\tilde{Y}^*$ are smooth away from the torus-fixed points. If $\tilde{X}$ (respectively $\tilde{X}^*$) denotes the closure 
of $X^\circ$ (respectively $(X^*)^\circ$) in $\tilde{Y}$ (respectively $\tilde{Y}^*$), then the induced morphisms $\tilde{X} \rightarrow X$ and $\tilde{X}^* \rightarrow X^*$ are $G$-equivariant, crepant resolutions of $X$ and
$X^*$ respectively. 

\begin{example}\label{e:boundary}
Since $P$ has a unique interior lattice point, Corollary~\ref{c:zerosmooth} implies that
\[
H^{d - 1,0}(\tilde{X}) = \det(\rho),  \: \: H^{0,0}(\tilde{X}) = 1,
\]
and 
\[
H^{p,0}(\tilde{X}) = 0 \textrm{ for } 0 < p < d - 1.
\]
By symmetry, this establishes Conjecture~\ref{c:mirror} along the boundary of the Hodge diamond. 
\end{example}

If $Q$ is a proper face of $P$, then we let $Q^*$ denote the dual face in $P^*$. Since $\dim Q + \dim Q^* = d - 1$, we have a bijection between edges of $P$ and codimension $2$ faces of $P^*$.  
We define
\[
\theta(P^*) = \sum_{\substack{[Q] \in P / G \\ \dim Q = 1}} \Ind_{G_Q}^G [ \det(\rho_Q) \cdot 
\chi^*_Q \cdot  \chi^*_{Q^*} ], 
\]
\[
\theta(P) = \sum_{\substack{[Q^*] \in P^* / G \\ \dim Q^* = 1}} \Ind_{G_{Q^*}}^G [ \det(\rho_{Q^*}) \cdot 
\chi^*_Q \cdot  \chi^*_{Q^*} ].
\]
Recall that $\Phi_k = \Phi(P)_k$ denotes the lattice points 
in $P$ which lie in the relative interior of a $k$-dimensional face of $P$. 
We now verify Conjecture~\ref{c:mirror} for two more pieces of the Hodge diamond.


\begin{corollary}
With the notation above, if $P$ is a reflexive polytope, and $X$ admits a crepant, toric resolution $\tilde{X}$, then, for $d \ge 3$, the non-primitive part of the $G$-representation $H^{1,1}(\tilde{X})$ 
equals
\[
H^{1,1}(\tilde{X}) = \theta(P^*) + \chi_{P^*} -  \chi_{\lef \Phi(P^*)_{d - 1} \rig} - \rho - 1, 
\]
and the primitive part of the $G$-representation $H^{d - 2,1}(\tilde{X})$ equals
\[
H^{d - 2,1}(\tilde{X}) = \det(\rho) \cdot  [\theta(P) +  \chi_{P} -  \chi_{\lef \Phi(P)_{d - 1} \rig} - \rho - 1 ]. 
\]
\end{corollary}
\begin{proof}
Let $\tilde{Y} = \tilde{Y}(\Sigma)  \rightarrow Y = Y(\triangle)$ be an equivariant, crepant, toric morphism inducing $\tilde{X} \rightarrow X$. Here $\triangle$ is the fan over the faces of $P^*$, and 
 the rays of 
$\Sigma$ not lying in the interior 
of a maximal cone of $\triangle$ are in bijection with the lattice points on the boundary of $P^*$ 
not lying in the relative interior of a facet of $P$. 
Corollary~\ref{c:compute2} implies that the non-primitive part of the $G$-representation 
$H^{1,1}(\tilde{X})$  equals
\[
\theta(P^*) + \sum_{k = 0}^{d - 2} \chi_{\lef \Phi(P^*)_{k} \rig}  - \rho.  
\]
Since $P^*$ contains a unique interior lattice point, the latter sum is equal to 
\[
\theta(P^*) + \chi_{P^*} - 1 - \chi_{\lef \Phi(P^*)_{d - 1} \rig}  - \rho, 
\]
as desired. 

On the other hand, by Corollary~\ref{c:real}, the primitive part of the $G$-representation $H^{d- 2,1}(\tilde{X})$ equals 
\[
\det(\rho) \cdot [ \phi_{P,d - 1} -  \chi_{\lef \Phi(P)_{d - 1} \rig} + \theta(P)].
\]
By Corollary~\ref{c:reflexive}, $\phi_{P,d - 1} = \phi_{P,1} = \chi_P - \rho - 1$. 
This completes the proof. 
\end{proof}

As an immediate consequence we obtain a positive answer to Conjecture~\ref{c:mirror} in the case when 
$X$ and $X^*$ admit toric, crepant $G$-equivariant resolutions, and $\dim X \le 3$.  

\begin{corollary}\label{c:dim3}
Let $P$ and $P^*$ be polar, $G$-invariant, reflexive polytopes of dimension $d \le 4$, and let $X$ and $X^*$ be corresponding $G$-invariant, non-degenerate hypersurfaces. If there exist $G$-equivariant, crepant, toric resolutions
 $\tilde{X} \rightarrow X$ and $\tilde{X}^* \rightarrow X^*,$ then 
 \[
H^{p,q}(\tilde{X}) = \det(\rho) \cdot H^{d - 1 - p,q}(\tilde{X}^*) \in R(G) \: \textrm{  for  } \: 0 \le p,q \le d - 1. 
\]
\end{corollary}

\bibliographystyle{amsplain}
\bibliography{alan}

\end{document}